%% file: aglv1.tex
\documentclass[11pt]{article}
\usepackage[english]{babel}
\usepackage{latexsym, amsfonts, amsthm, amsmath, enumerate,dsfont}
\usepackage{graphicx,color}

\title{Stability and convergence in discrete convex monotone dynamical 
systems}

\author{Marianne Akian, St\'ephane Gaubert \footnote{The first two authors were supported by the Arpege  programme of the French National Agency of Research (ANR), project ``ASOPT'', number ANR-08-SEGI-005 and by the Digiteo project DIM08 ``PASO'' number 3389.} \\
\small INRIA and CMAP, \'Ecole Polytechnique,\\
\small  91128 Palaiseau Cedex, France.\\
\small {\tt marianne.akian@inria.fr}, {\tt stephane.gaubert@inria.fr}\\
\mbox{}\\
and \\
\mbox{}\\
Bas Lemmens,\\
\small SMSAS, University of Kent,\\
\small Canterbury, CT2 7NF, United Kingdom\\
\small {\tt B.Lemmens@kent.ac.uk}}

\let\phi=\varphi
\let\theta=\vartheta
\let\epsilon=\varepsilon

\newtheorem{theorem}{Theorem}[section]
\newtheorem{Definition}[theorem]{Definition}
\newtheorem{lemma}[theorem]{Lemma}
\newtheorem{corollary}[theorem]{Corollary}
\newtheorem{proposition}[theorem]{Proposition}
\newtheorem{hypothesis}[theorem]{Hypothesis}
\newenvironment{definition}{\begin{Definition}\rm}{\end{Definition}}

\begin{document}
\maketitle

\begin{abstract}
We study the stable behaviour of discrete dynamical systems where the map is convex and monotone with respect to the standard positive cone. The notion of tangential stability for fixed points and periodic points is introduced, which is weaker than Lyapunov stability. Among others we show that the set of tangentially stable fixed points is isomorphic to a convex inf-semilattice, and a criterion is given for the existence of a unique tangentially stable fixed point. We also show that periods of tangentially stable  periodic points are orders of permutations on $n$ letters, where $n$ is the dimension of the underlying space, and a sufficient condition for global convergence to periodic orbits is presented.     
\end{abstract} 

\section{Introduction}
Many natural dynamical systems preserve a type of ordering on the state space. 
Such dynamical systems are called monotone and often display rather simple behaviour. In the last couple of decades monotone dynamical systems have been studied intensively, see \cite{HS} for an up-to-date survey. Ground breaking work on monotone dynamical systems was done by Hirsch \cite{H1,H2}, who showed, among others, that  in a continuous time strongly monotone dynamical system almost all pre-compact orbits converge to the set of equilibrium points. In a discrete time strongly monotone dynamical system one has generic convergence to periodic orbits under appropriate conditions on the map, see \cite{DH,HessPol,PT1}. 
Various additional conditions have been  studied to obtain convergence of all orbits instead of almost all orbits. A type of concavity condition, also called sub-homogeneity, has received a great deal of attention, see  \cite{AGLN,KN,LSp,Sm,T2}. The concavity condition makes the dynamical system non-expansive, which allows one to prove strikingly detailed results concerning their behaviour.  

In this paper we study discrete time dynamical systems 
\begin{equation}\label{eq:1.1}
x^{k+1}=f(x^k)\mbox{\quad for }k=0,1,2,\ldots, 
\end{equation}
where $f\colon\mathcal{D}\to\mathcal{D}$ is a convex monotone map on $\mathcal{D}\subseteq\mathbb{R}^n$ preserving the partial ordering induced by the standard positive cone. Such dynamical systems are in general not non-expansive. 
We introduce the notion of tangential stability for fixed points and periodic points, which is weaker than Lyapunov stability. It turns out that this notion is the right one to prove a variety of results concerning the stable behaviour of monotone convex dynamical systems, which are of comparable detail as the ones for monotone non-expansive dynamical systems. 

In particular, we show that the tangentially stable fixed point set is isomorphic to a convex inf-semilattice in $\mathbb{R}^n$.  We also give a criterion for the existence of a  unique  tangentially stable fixed point. In addition, tangentially stable periodic orbits are analysed and a condition is presented under which there is global convergence to Lyapunov stable periodic orbits. Among others it is shown that the periods of tangentially stable periodic points divide the cyclicity of the critical graph, which implies that the periods are orders of permutations on $n$ letters. However, the periods of unstable periodic orbits can be arbitrary large. 

The results are a continuation of \cite{AG} in which the first two authors  studied discrete time dynamical systems (\ref{eq:1.1}), where 
$f\colon\mathbb{R}^n\to\mathbb{R}^n$ is not only convex and monotone, but also  additively sub-homogeneous. 
The extra sub-homogeneity condition makes the dynamical system non-expansive under the sup-norm \cite{CT}. The non-expansiveness property severely constrains 
the complexity of its behaviour \cite{LS2,N1} and  makes all fixed points and periodic orbits  Lyapunov stable. It also ensures that the subdifferential of $f$ at a fixed point consists of row-stochastic matrices. Without the additively homogeneity condition, the subdifferential merely consists of stable nonnegative matrices, which makes the analysis more subtle. 

Motivating examples of discrete  convex monotone dynamical systems arise  in Markov decision processes and game theory as value iteration schemes, see \cite{AG} and the references therein. The results in this paper extend results for Markov decision processes with  sub-stochastic transition matrices to arbitrary nonnegative matrices. In particular, they apply to Markov decision processes with negative discount rates, see \cite{Roth2}.  Discrete convex monotone dynamical systems are also  used in static analysis of programs by abstract interpretation \cite{goubault}, i.e., automatic verification  of variables in computer programs. 
They also appear in  the theory of discrete event systems \cite{BCOQ}, statistical mechanics \cite{N2}, and in the analysis of imprecise Markov chains \cite{dC}. At the end of Section 2  we give several explicit examples.

The paper contains nine sections. In Section 2  several basic definitions and properties of convex monotone maps are collected. 
Subsequently various degrees of stability of fixed points of convex monotone maps are discussed and the notion of tangential stability is introduced. 
In Section 4 several preliminary results concerning stable nonnegative matrices are given. Among others rectangular sets of stable nonnegative matrices are studied. 
In Section 5 we analyse tangentially stable fixed points and introduce the critical graph of a monotone convex map. Section 6 is used to collect several preliminary results concerning convex monotone positively homogeneous maps that are needed in the analysis of the geometry of the tangentially stable fixed point set. 
Section 7 contains the main result on the geometry of the tangentially stable fixed point set. Section 8 concerns tangentially stable periodic points and their periods. In the final section  a criterion is given under which each orbit of a discrete time convex monotone dynamical systems converges to a Lyapunov stable periodic orbit.  

\section{Basic properties of convex monotone maps} 
Let $\mathbb{R}^n_+=\{x\in\mathbb{R}^n\colon x_i\geq 0\mbox{ for } 1\leq i\leq n\}$ denote 
the \emph{standard positive cone}. 
The cone $\mathbb{R}^n_+$ induces a partial ordering on $\mathbb{R}^n$ by $x\leq y$ 
if $y-x\in\mathbb{R}^n_+$. 
We write $x\ll y$ if $y-x$ is in the interior of $\mathbb{R}^n_+$. 
In particular, we say that $x$ is \emph{positive} if $0\ll x$. 
For $x,y\in\mathbb{R}^n$ we also use the notation $x\geq y$ and $x\gg y$ with the obvious interpretation. 
A set $\mathcal{X}\subseteq\mathbb{R}^n$ is called 
\emph{bounded from above} if there exists 
$u\in\mathbb{R}^n$ such that $x\leq u$ for all $x\in  \mathcal{X}$. 
Similarly, we say that $\mathcal{X}\subseteq\mathbb{R}^n$ is 
\emph{bounded from below} if there 
exists $l\in\mathbb{R}^n$ such that $l\leq x$ for all $x\in \mathcal{X}$. 
The partially ordered vector space $(\mathbb{R}^n,\leq)$ is a vector lattice, where the 
binary relations $\wedge$ and $\vee$ are defined as follows. For $x,y\in
(\mathbb{R}^n,\leq)$,  $x\wedge y$ is the greatest lower bound of $x$ and $y$, so $(x\wedge y)_i=\min\{x_i,y_i\}$ for $1\leq i\leq n$, and $x\vee y$ is least upper bound of $x$ and $y$, so $(x \vee y)_i=\max\{x_i,y_i\}$ for $1\leq i\leq n$. 
  
A map $f\colon \mathcal{D}\to\mathbb{R}^n$, where $\mathcal{D}\subseteq\mathbb{R}^m$, is called \emph{monotone} if for each $x,y\in\mathcal{D}$ with $x\leq y$ we have that $f(x)\leq f(y)$.  It is called \emph{strongly monotone} if $x\leq y$ and $x\neq y$ implies that $f(x)\ll f(y)$. 
A map $f\colon \mathcal{D}\to\mathbb{R}^n$, where 
$\mathcal{D}\subseteq\mathbb{R}^m$ is convex, is called \emph{convex} if 
\[
f(\lambda x+(1-\lambda)y)\leq \lambda f(x)+(1-\lambda)f(y)\mbox{\quad for all } 
0\leq \lambda\leq 1\mbox{ and }x,y\in\mathcal{D}. 
\]
In other words, $f\colon\mathcal{D}\to\mathbb{R}^n$ is convex if each coordinate function is convex in the usual sense. 
The reader may note that our notion of monotonicity is different from the one commonly used in convex analysis \cite{Rock}. 

The \emph{orbit} of $x\in\mathcal{D}$ under a map $f\colon\mathcal{D}\to\mathcal{D}$ is given by $\mathcal{O}(x;f)=\{f^k(x)\colon k=0,1,2,\ldots\}$. We say that $x\in\mathcal{D}$ is a \emph{periodic point} of $f\colon\mathcal{D}\to\mathcal{D}$ if $f^p(x)=x$ for some integer $p\geq 1$, and the minimal such $p\geq 1$ is called the \emph{period} of $x$ under $f$. 

Let  $\mathcal{M}_{m,n}$ denote the set of all $m\times n$ real matrices, and let 
$\mathcal{P}_{m,n}$  be the set of all nonnegative matrices in $\mathcal{M}_{m,n}$. 
A matrix $P\in \mathcal{P}_{m,n}$ is called \emph{positive} if $p_{ij}>0$ for all 
$1\leq i\leq m$ and $1\leq j\leq n$. 
Given a matrix $M\in \mathcal{M}_{m,n}$ we denote its rows by 
$M_1, \ldots, M_m\in\mathcal{M}_{1,n}$, and we identify $M$ with the $m$-tuple 
$(M_1,\ldots,M_m)$. So, $\mathcal{M}_{m,n}$ is identified with the $m$-fold direct product 
\[
\mathcal{M}_{m,n}= \mathcal{M}_{1,n}\times \ldots\times\mathcal{M}_{1,n}.
\]
We say that $\mathcal{R}\subseteq \mathcal{M}_{m,n}$ is \emph{rectangular} if 
$\mathcal{R}$ can be written as 
$\mathcal{R}=\mathcal{R}_1\times\ldots\times\mathcal{R}_m$, where $\mathcal{R}_1,\ldots,\mathcal{R}_m$ are non-empty subsets of $\mathcal{M}_{1,n}$. 
Furthermore it is convenient to introduce the following matrix notation. 
Given $M\in\mathcal{M}_{m,n}$, $I\subseteq\{1,\ldots,m\}$, and $J\subseteq\{1,\ldots,n\}$ we write $M_{IJ}$ to denote the $|I|\times|J|$ sub-matrix of $M$ with row indices in $I$ and column indices in $J$. 
Likewise, given $x\in\mathbb{R}^n$ and $K\subseteq\{1,\ldots,n\}$ we write $x_K\in\mathbb{R}^K$ to denote the vector in $\mathbb{R}^K$ obtained by restricting $x$ to its coordinates in $K$. 

For  a convex map $f\colon \mathcal{D}\to\mathbb{R}^n$, where $\mathcal{D}\subseteq\mathbb{R}^m$ is open and convex, the \emph{subdifferential of $f$ at $v\in\mathcal{D}$} is defined by,
\begin{equation}\label{eq:2.1}
\partial f(v)=\{M\in\mathcal{M}_{m,n}\colon f(x)-f(v)\geq M(x-v)\mbox{ for all }
x\in\mathcal{D}\}. 
\end{equation}
In the following proposition  several basic facts concerning the subdifferential 
are collected, cf.\ \cite[Theorem 23.4]{Rock}. 
\begin{proposition}\label{prop:2.1}
If $f$ is a convex map from an open convex subset $\mathcal{D}\subseteq\mathbb{R}^m$ to $\mathbb{R}^n$, then for each $v\in\mathcal{D}$, the set
$\partial f(v)$ is non-empty, compact, convex and rectangular.   
\end{proposition}
We note that the rectangularity of $\partial f(v)$ follows directly from the fact that 
$f(x)-f(v)\geq M(x-v)$ is equivalent to $f_i(x)-f_i(v)\geq M_i(x-v)$ for all $1\leq i\leq m$. 
If, in addition, the map is monotone, then $\partial f(v)$ consists of nonnegative matrices as the following proposition shows. 
\begin{proposition}\label{prop:2.2} 
If $f\colon\mathcal{D}\to\mathbb{R}^n$, where $\mathcal{D}\subseteq\mathbb{R}^m$ is open and convex, is a convex monotone map, then $\partial f(v)\subseteq \mathcal{P}_{m,n}$ for all $v\in\mathcal{D}$. 
Moreover, if $f$ is strongly monotone, then each $P\in\partial f(v)$ is positive.
\end{proposition}
\begin{proof}
If $v\in\mathcal{D}$, then there exists $\mathcal{U}$ open neighbourhood of $0$ such that $v-u \in \mathcal{D}$ for all $u\in\mathcal{U}$. Now if $u\geq 0$, with $u\in\mathcal{U}$, and  $P\in\partial f(v)$, then 
$0\geq f(v-u)-f(v)\geq -Pu$, as $f$ is monotone. 
Thus, $Px\geq 0$ for all $x\in\mathbb{R}^n_+$ and hence $P\in\mathcal{P}_{m,n}$.
We note that if $f$ is strongly monotone, $u\geq 0$ and $u\neq 0$, then $0\gg f(v-u)-f(v)\geq -Pu$. 
This implies that $Px\gg 0$ for all $x\in\mathbb{R}^n_+\setminus \{0\}$ and 
therefore $P$ is positive. 
\end{proof}

For a convex map $f\colon \mathcal{D}\to\mathbb{R}^n$, where $\mathcal{D}\subseteq\mathbb{R}^m$ is open and convex, and $v\in\mathcal{D}$, the \emph{one-sided directional derivative of $f$ at $v$} is given by, 
\begin{equation}\label{eq:2.2} 
f'_v(y)=\lim_{\epsilon\downarrow 0}\frac{f(v+\epsilon y)-f(v)}{\epsilon}.
\end{equation}
The map $f'_v\colon\mathbb{R}^m\to\mathbb{R}^n$ is well-defined, convex, finite 
valued and positively homogeneous,
meaning that $f'_v(\lambda x)=\lambda f'_v(x)$ for all $\lambda >0$ and $x\in\mathbb{R}^m$, see \cite[Theorem 23.1]{Rock}. 
Moreover, $f'_v$ is monotone (because it is defined as a pointwise limit of monotone maps). We shall occasionally need the following representation of $f'_v$:
\begin{equation}\label{eq:2.3}
f'_v(y)=\sup_{P\in\partial f(v)} Py\mbox{\quad for }y\in\mathbb{R}^m
\end{equation}
(see \cite[Theorem 23.4]{Rock}). 
If $f\colon\mathcal{D}\to\mathcal{D}$, where $\mathcal{D}\subseteq\mathbb{R}^n$ is open and convex, then for each $v\in\mathcal{D}$ we also have that 
\[
(f'_v)^k=f'_{f^{k-1}(v)} \circ \cdots \circ f'_{f(v)}\circ f'_v
\]
(see \cite[Lemma 4.3]{AG}). 

Throughout the remainder of the exposition we shall make the following assumption on the domain of convex monotone maps $f\colon\mathcal{D}\to\mathcal{D}$. 
\begin{hypothesis} \label{hyp:2.3}
The set $\mathcal{D}\subseteq\mathbb{R}^n$ is open and convex. 
\end{hypothesis}
Although in some results more general domains can be treated, we restrict ourselves to this case, as it simplifies the presentation.

To conclude this section we briefly discuss several examples of convex monotone maps. In the theory of Markov decision processes one considers monotone convex maps $f\colon\mathbb{R}^n\to\mathbb{R}^n$ of the form: 
\begin{equation} \label{sup}
f_i(x)= \sup_{j\in A_i} r^j_i +p^j\cdot x\mbox{\quad for }i=1,\ldots,n.
\end{equation}
Here $r^j_i\in\mathbb{R}$ and $p^j$ is a sub-stochastic vector for each $j\in A_i$.  
The results in this paper apply to the case where $p^j$ is merely a nonnegative vector.  

Other examples arise in the study of systems of polynomial equations, $x=P(x)$, where $P=(P_1,\ldots,P_n)$ and each $P_i$ is a polynomial with variables $x_1,\ldots,x_n$ and  nonnegative coefficients. Looking for a 
positive solution of $x=P(x)$ is equivalent to finding
a fixed point of the map $f(x)=\mathrm{Log} \circ P\circ \mathrm{Exp}$,
where $\mathrm{Exp}(x_1,\ldots,x_n)=(e^{x_1},\ldots,e^{x_n})$ 
and $\mathrm{Log}$ denotes its inverse.  We can write 
$P_i(x)=\sum_{j\in A_i}a_{ij}x^j$,
where $A_i\subseteq\mathbb{N}^n$ is a finite set and each $a_{ij}\geq 0$, with the convention that $x^j=x_1^{j_1}\cdots x_n^{j_n}$
for $j=(j_1,\ldots,j_n)\in\mathbb{N}^n$. Then   
\begin{align}
f_i(x)=\log(\sum_{j\in A_i}a_{ij}\exp( j\cdot x)).
\label{e-posynomial}
\end{align}
Such ``log-exp'' functions are not only monotone, but also convex, see \cite[Example 2.16]{rockwets}. More generally, we could allow $A_i$ to be a subset $\mathbb{R}_+^n$, instead of $\mathbb{N}^n$. This yields a class of functions $P_i$ which are usually   called {\em posynomials} \cite{boyd}. Posynonmials  play a role in  static analysis of programs by abstract interpretation \cite{goubault}.
Note that the example in (\ref{sup}) can be  obtained as a limit of posynomials
by setting $a_{ij}=e^{\beta r_i^j}$ and
\[
f^\beta_i(x):=\beta^{-1}\log(\sum_{j\in A_i}\exp(\beta(r_i^j+ p^j\cdot x))).
\]
If $\beta$ tends to $+\infty$, then $f^\beta_i$ converges to the map (\ref{sup}), see \cite{viro}.

\section{Stability of fixed points}
Recall that a fixed point $v\in\mathcal{D}$ of $f\colon\mathcal{D}\to\mathcal{D}$ is 
\emph{Lyapunov stable} if for each neighbourhood $\mathcal{U}$ of $v$, 
there exists a neighbourhood $\mathcal{V}\subseteq \mathcal{D}$ of $v$ such that  
$v\in\mathcal{V}$ and $f^k(\mathcal{V})\subseteq \mathcal{U}$ for all $k\geq 1$. 
An $n\times n$ matrix $P$ is called \emph{stable} if all the orbits of $P$ are bounded. 
Stable matrices have the following well-known characterizations.
\begin{proposition}\label{prop:4.1}
For a matrix $P$ the following assertions are equivalent:
\begin{enumerate}[(i)]
\item $P$ is stable.
\item There exists a norm on $\mathbb{R}^n$ such that the induced matrix norm of $P$ is at most one.
\item All the eigenvalues of $P$ have modulus at most one and all the eigenvalues of modulus one are semi-simple.
\item The origin is a Lyapunov stable fixed point of $P$. 
\end{enumerate}
\end{proposition}
In the analysis of convex monotone dynamical systems, it is useful to distinguish 
various notions of stability that are weaker than the classical Lyapunov stability. 
\begin{proposition}\label{prop:3.1}
 If $f\colon\mathcal{D}\to\mathcal{D}$ is a convex monotone map  with a fixed point 
 $v\in\mathcal{D}$, then the following assertions; 
 \begin{enumerate}[(i)]
 \item $v$ is a Lyapunov stable fixed point, 
 \item there exists a neighbourhood $\mathcal{V}\subseteq\mathcal{D}$ of $v$ such that 
 every orbit of $x\in\mathcal{V}$ is bounded,
 \item there exists a neighbourhood $\mathcal{V}\subseteq\mathcal{D}$ of $v$ such that  every orbit of $x\in\mathcal{V}$ is bounded from above, 
\item every orbit of $f'_v\colon\mathbb{R}^n\to\mathbb{R}^n$ is bounded,
\item every orbit of $f'_v\colon\mathbb{R}^n\to\mathbb{R}^n$ is bounded from above,
\item each $P\in\partial f(v)$ is stable. 
\item every orbit of $f'_v\colon\mathbb{R}^n\to\mathbb{R}^n$ is bounded from below,
\item every orbit of $f\colon\mathcal{D}\to\mathcal{D}$ is bounded from below,
\end{enumerate}
 satisfy the following implications: 
 \[
 (i)\Leftrightarrow (ii)\Leftrightarrow (iii)\Rightarrow (iv)\Leftrightarrow (v)\Rightarrow (vi) 
 \Rightarrow (vii)\Rightarrow (viii). 
 \]
\end{proposition}
\begin{proof}
The implications $(i)\Rightarrow (ii)\Rightarrow (iii)$ and $(iv)\Rightarrow (v)$ are trivial. 
We start by showing that $(iii)$ implies $(v)$. Let $x\geq 0$ be such that 
$v+x\in\mathcal{V}$. 
Remark that $f(v+x)\geq f(v)+f'_v(x)=v+f'_v(x)$. Since $f$ is monotone, we deduce that 
\begin{equation}\label{3.1}
f^k(v+x)\geq v+(f'_v)^k(x)\mbox{\quad for all }k\geq 1,
\end{equation}
and hence $\mathcal{O}(x; f'_v)$ is bounded from above. 
As $f'_v$ is monotone and positively homogeneous, there exist for each $y\in\mathbb{R}^n$, a vector $x\geq 0$ and a scalar $\lambda>0$ such that $y\leq \lambda x$ and 
$v+x\in\mathcal{V}$. Thus, we get that  $f_v'$ has all its orbits bounded from above. 

Next we prove that $(v)$ implies $(vi)$. 
If  $P\in\partial f(v)$, then we  know by (\ref{eq:2.3})  that  $f'_v(x)\geq Px$. 
As $f'_v$ is monotone, we get that $(f'_v)^k(x)\geq P^kx$ for all $k\geq 1$ and $x\in\mathbb{R}^n$ and therefore $P$ has all its orbits bounded from above. This implies that $P$ has all its orbits bounded, since $P$ is linear. 

Suppose that $\partial f(v)$ contains a stable matrix $P$. We deduce from (\ref{eq:2.3}) 
that $(f'_v)^k(x)\geq P^kx$ for all $k\geq 1$ and $x\in\mathbb{R}^n$. As $P$ is stable,  this implies that $\mathcal{O}(x; f'_v)$ is bounded from below, which shows that 
$(vi)$ implies $(vii)$. 

To see that $(vii)$ implies $(viii)$ let $x\in\mathcal{D}$ and note that, by (\ref{3.1}), 
\[
f^k(x)\geq v+(f'_v)^k(x-v)\mbox{\quad for all }k\geq 1.
\]
 As $\mathcal{O}(x-v;f'_v)$ is bounded from below, we get that $\mathcal{O}(x;f)$ is 
also bounded from below. 

Note that $(v)$ implies $(vi)$ and $(vi)$ implies $(vii)$, so that $(iv)$ and $(v)$ 
are equivalent. Similarly, $(iii)$ implies $(viii)$, so that $(ii)$ and $(iii)$ are equivalent. 
It remains to be shown that $(ii)$ implies $(i)$. 

If $(ii)$ holds there exists $u\gg 0$ such that $v+\lambda u\in\mathcal{D}$ for all 
$|\lambda|\leq 1$ and $\mathcal{O}(v+u;f)$ is bounded. 
Let $k\geq 1$ and note that for $0<\lambda <1$,
 \begin{eqnarray*}
\lambda(f^k(v+u) -v) & = & \lambda(f^k(v+u) - v) + 
(1-\lambda) (f^k(v)-v)\\
   & \geq & f^k(\lambda(v+u) + (1-\lambda) v) - v\\
   & = & f^k(v+\lambda u) -v \enspace .
 \end{eqnarray*}
 Take $P\in\partial f(v)$ and remark that $f(x)\geq P(x-v) +v$. 
 This implies that $f^k(x)\geq P^k(x-v)+v$ for all $x\in\mathcal{D}$. 
 Hence 
 \begin{equation}\label{eq:3.3} 
 f^k(v-\lambda u)-v\geq -\lambda P^ku\mbox{\quad for }0<\lambda <1.
 \end{equation}
 As $(ii)$ implies $(vi)$, we know that $P$ is stable. 
 Define a norm $\|\cdot\|_u$ on $\mathbb{R}^n$ by $\|x\|_u=\inf\{\alpha>0\colon -\alpha u\leq x\leq \alpha u\}$ and let \[
 \gamma=\sup_{k\geq 1} \{\|P^k u\|_u,\|f^k(v+u)-v\|_u\}.
 \] 
 Note that $\gamma<\infty$, as $P$ is stable and $\mathcal{O}(v+u;f)$ is bounded. 
Let $\lambda:=\|x-v\|_u$, so that $v-\lambda u\leq x\leq v+\lambda u$.
If $\lambda < 1$, 
we get that $-\lambda \gamma u\leq f^k(x)-v\leq \lambda\gamma u$ for each $k\geq 1$. 
Thus, $\|f^k(x)-v\|_u\leq \gamma\lambda= \gamma\|x-v\|_u$ for all $x\in\mathcal{D}$ 
such that $\|x-v\|_u<1$ and for all $k\geq 1$. Hence $v$ is a Lyapunov stable fixed point. 
\end{proof}

\paragraph{Example 1.} Let $f\colon\mathbb{R}^n\to\mathbb{R}^n$ be given by 
$f(x)=\max\{0,x+x^2\}$ for $x\in\mathbb{R}^n$. Then $f'_0(x)=\max\{0,x\}$ and hence  every orbit of $f'_0$ is bounded. However, the orbit of each $x>0$ is unbounded under 
$f$. This show that $(iv)$ does not imply $(iii)$. 

\paragraph{Example 2.} This example shows that $(vi)$ does not imply $(v)$. 
Consider $h(p)=-p\log p -(1-p)\log (1-p)$ for $p\in [0,1]$, and define $g\colon\mathbb{R}^2\to\mathbb{R}^2$ by 
\[
g(x)=\sup_{p\in [0,1]} px_1+h(p)x_2\mbox{\quad for }x=(x_1,x_2)\in\mathbb{R}^2,
\]
which is the Legendre-Fenchel transform of $-h$. 
Note that $g(x)=x_2\log (1+e^{x_1/x_2})$, for $x_2>0$. 
Indeed, put $x_2=1$ and consider 
\[
\frac{d}{dp} (px_1+h(p))=0.
\]
Solving for $p$ gives $p= e^{x_1}/(1+e^{x_1})$, so that $g(x_1,1)=\log (1+e^{x_1})$. 
As $g$ is positively homogeneous,  $g(x)=x_2\log (1+e^{x_1/x_2})$ for $x_2>0$. 

Now define $f\colon \mathbb{R}^2\to\mathbb{R}^2$ by 
\[
f_1(x)=\left\{\begin{array}{cl}
                   \max\{0, x_1\} & \mbox{ if } x_2\leq 0\\
                   x_2\log (1+e^{x_1/x_2}) & \mbox{ if }x_2>0
                   \end{array}\right.
\]
and $f_2(x)=x_2$ for all $x=(x_1,x_2)\in\mathbb{R}^2$. 
If $x_2>0$, we get that $f^k_1(x)=x_2\log(k+e^{x_1/x_2})\to\infty$, as $k\to\infty$. 
Thus, not all orbits of $f'_0=f$ are bounded from above. But $\partial f(0)$ consists 
of matrices of the form
\[
\left(\begin{matrix}
p & s\\
0 & 1
\end{matrix}\right), \mbox{\quad where }0\leq s\leq h(p)\mbox{ and } 0\leq p\leq 1. 
\] 
As $h(1)=0$, all these matrices are stable. 

\paragraph{Example 3.} To see that $(vii)$ does not imply $(vi)$ we consider the map 
$f\colon\mathbb{R}\to\mathbb{R}$ given by $f(x)=\max\{0,2x\}$ for $x\in\mathbb{R}$. 
Then $f'_0=f$, so every orbit of $f'_0$ is bounded from below, but $2\in\partial f(0)$.  

\paragraph{Example 4.} To prove that $(viii)$ does not imply $(vii)$ consider $f\colon\mathbb{R}^2\to\mathbb{R}^2$ given by 
\[
f\left(\begin{matrix} x_1\\x_2\end{matrix}\right) = \left(\begin{matrix} 
e^{x_1}+e^{x_2} - 2\\ x_2\end{matrix}\right) \mbox{\quad for } 
x=(x_1,x_2)\in\mathbb{R}^2.
\]
Then 
\[
f'_0\left(\begin{matrix} x_1\\ x_2\end{matrix}\right) =\left(\begin{matrix} 1 & 1 \\ 0 & 1 \end{matrix}\right)\left(\begin{matrix} x_1\\x_2\end{matrix}\right)\mbox{\quad for all } 
x=(x_1,x_2)\in\mathbb{R}^2,
\]
so that $\mathcal{O}(-u;f'_0)$ is unbounded from below for $u\gg 0$. 
But clearly every orbit of $f$ is bounded from below. 

We use the following notion of stability which is weaker than ordinary Lyapunov stability according to Proposition \ref{prop:3.1}. 
\begin{definition}\label{def:3.2} 
Let $f\colon \mathcal{D}\to\mathcal{D}$ be a convex monotone map and  $v\in\mathcal{D}$ be a fixed point of $f$. We say that $v$ is a \emph{tangentially stable}, or,  
\emph{t-stable}, fixed point if $f'_v$ has all its orbits bounded from above. Similarly, we call a periodic point $\xi\in\mathcal{D}$ with period $p$  \emph{tangentially stable} if it is a t-stable fixed point of $f^p$. 
\end{definition}
We note that if $v$ is a t-stable periodic point of $f$ with period $p$, then $f^m(v)$ is also  t-stable for each $0<m<p$. Indeed, as 
\[
(f^p)'_v =(f^{p-m})'_{f^m(v)}\circ (f^m)'_v\mbox{\quad and\quad }(f^p)'_{f^m(v)} =(f^m)'_v\circ (f^{p-m})'_{f^m(v)},
\] 
we get that 
\[
((f^p)'_{f^m(v)})^k = ((f^m)'_v\circ (f^{p-m})'_{f^m(v)})^k = (f^m)'_v\circ ((f^{p})'_v)^{k-1}\circ (f^{p-m})'_{f^m(v)}.
\]
Thus, every point in the orbit of a t-stable periodic point is also t-stable. 
We denote by $\mathcal{E}(f)$ the set of all fixed points of $f$ and we let 
\[
\mathcal{E}_t(f)=\{v\in\mathcal{E}(f)\colon v\mbox{ is t-stable}\}. 
\]
The subdifferential of a t-stable fixed point consists of nonnegative stable matrices by Proposition \ref{prop:3.1}. In the next section we collect  some results concerning stable matrices that will be useful in the analysis. 
 
\section{ Stable nonnegative matrices} 
To an $n\times n$  nonnegative matrix $P=(p_{ij})$ we associate a directed graph 
$\mathcal{G}(P)$ on $n$
nodes, in the usual way, by letting an arrow go from node $i$ to $j$ if $p_{ij}>0$.
We say that a node $i$ has \emph{access} to a node $j$ if there is a (directed) path in 
$\mathcal{G}(P)$ from $i$ to $j$. 
Using the notion of access one defines an equivalence relation 
$\sim$ on $\{1,\ldots,n\}$ by $i\sim j$ if $i$ has access to $j$ and vice versa. 
The equivalence classes are called \emph{classes of }$P$.
A nonnegative matrix $P$ is called $\emph{irreducible}$ if it has only one class. 
Otherwise it is said to be \emph{reducible}. 

The \emph{spectral radius} of $P$ is given by $\rho(P)=\max\{|\lambda|\colon \lambda\mbox{ eigenvalue of } P\}$. 
If $P$ is a stable nonnegative matrix, we call a class $C$ of $P$ \emph{critical} 
if $\rho(P_{CC})=1$. Recall that in Perron-Frobenius theory 
a class $C$ of $P$ is called \emph{basic} if $\rho(P_{CC})=\rho(P)$. 
Thus, all critical classes of a stable nonnegative matrix are basic by
Proposition \ref{prop:4.1}.
The following proposition is a direct consequence of \cite[Theorem 3]{Roth} (see also 
\cite[Corollary 3.4]{Schnei}) and Proposition \ref{prop:4.1}. 
\begin{proposition}\label{prop:4.2}
If $P$ is a nonnegative stable matrix and $C$ and $C'$ are two distinct critical 
classes of $P$, then no $i\in C$ has access to any $j\in C'$.  
\end{proposition} 
Moreover, we have the following general fact concerning nonnegative matrices (cf.\ 
\cite[Chapter XIII \S3.4]{Gant}. 
\begin{proposition}\label{prop:4.2bis}
If $P$ is a nonnegative matrix, then $\rho(P')\leq \rho(P)$ for all principal
submatrices $P'$ of $P$. 
If $P$ is reducible, then the equality holds for at least one principal submatrix $P'$ of $P$ with 
$P'\neq P$.
\end{proposition}
Using these proposition we now prove the following useful normal form for stable nonnegative matrices.
\begin{proposition}\label{prop:4.3}
If $P$ is a nonnegative stable $n\times n$ matrix, then there exist a permutation 
matrix $\Pi$ and a unique partition of $\{1,\ldots,n\}$ into disjoint sets 
$U$, $C$, $D$, and $I$ such that
\begin{enumerate}[(i)]
\item
\begin{equation}\label{eq:4.1}
\Pi^T P \Pi  = \left ( \begin{array}{cccc}
                         P_{UU} & P_{UC} & P_{UD} & P_{UI}\\
                          0     & P_{CC} & P_{CD} & 0 \\
                          0     & 0      & P_{DD} & 0 \\
                          0     & 0      & P_{ID} & P_{II}
                        \end{array} \right ),
\end{equation}
where $C$ is the disjoint union of the critical classes $C_1,\ldots,C_r$ of $P$,
\item $P_{CC}$ is block-diagonal, with blocks $P_{C_iC_i}$ for $1\leq i\leq r$,
\item every $i\in U$ has access to some $j\in C$, and for every $i\in D$ there
exists $j\in C$ that has access to $i$,
\item $I=\{1,\ldots,n\}\setminus(U\cup C\cup D)$.
\end{enumerate}
Moreover, in that case, we that $\rho(P_{UU})<1$, $\rho(P_{DD})<1$, and $\rho(P_{II})<1$.
\end{proposition}
\begin{proof} 
Let $C$ be the union of the critical classes $C_1,\ldots,C_r$ of $P$. 
Then $P_{CC}$ is block diagonal by Proposition \ref{prop:4.2}. 
Let $U$ be the set of nodes of $\mathcal{G}(P)$ not in $C$ that have access to some $i\in C$. 
In addition, let $D$ be the set of nodes $i$ in $\mathcal{G}(P)$ that are not in $C$, but from which there exists $j\in C$ such that $j$ has access to $i$. 
We remark that $U\cap D=\emptyset$. Indeed, if $i\in U\cap D$, then there exist $j_1,j_2\in C$ such that $i$ has access to $j_1$ and $j_2$ has access to $i$, By Proposition \ref{prop:4.2}, $j_1$ and $j_2$ are both in a single class, say $C_m$. 
But this implies that $i\in C_m$, which is a contradiction. 
In fact, the same argument shows that $P_{DU}=0$, $P_{CU}=0$, and $P_{DC}=0$. 

Now let $I=\{1,\ldots,n\}\setminus (U\cup C\cup D)$. By definition no node in $I$ has 
access to a node $C$, nor can it be accessed by a node in $C$. Hence $P_{IC}=0$ 
and $P_{CI}=0$. Furthermore, the definition of $U$ and $D$ implies that  $P_{IU}=0$ 
and $P_{DI}=0$. Thus, the sets $U$, $C$, $D$, and $I$, communicate as in Figure 
\ref{fig:4.1} and hence there exists a permutation matrix $\Pi$ such that 
$\Pi^T P\Pi$ satisfies (\ref{eq:4.1}). 
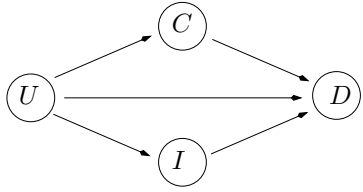
\begin{figure}
\[
\input{graph}
\]
\caption{Partition associated with a stable matrix}\label{fig:4.1}
\end{figure}

To prove the final assertion, we remark that if $\rho(P_{UU})=1$, then there exists 
a critical class $U^*\subseteq U$ such that $\rho(P_{U^*U^*})=\rho(P_{UU})=1$ and $P_{U^*U^*}$ is irreducible by Proposition \ref{prop:4.2bis}, which is a contradiction. In exactly the same way it can be shown that $\rho(P_{DD})<1$ and $\rho(P_{II})<1$. 
\end{proof} 
The set of nodes $U$, $C$, $D$, and $I$ in $\mathcal{G}(P)$ are respectively called 
\emph{upstream nodes}, \emph{critical nodes}, \emph{downstream nodes}, and 
\emph{independent nodes}. 
By using the normal form and the Perron-Frobenius theorem we now prove the 
following assertion. 
\begin{proposition}\label{prop:4.4}
If $P$ is nonnegative stable $n\times n$ matrix and $Pz\leq z$, then 
\begin{enumerate}[(i)]
\item $P_{CC}z_C=z_C$, 
\item $z_D=0$ and $(Pz)_{C\cup D}=z_{C\cup D}$,
\item $z_I\geq 0$, 
\item if, in addition, $z_S\geq 0$ for some $S\subseteq\{1,\ldots,n\}$ that contains at 
least one element of each critical class of $P$, then $z\geq 0$.
\end{enumerate} 
\end{proposition}
\begin{proof}
As $Pz\leq z$, it follows from Proposition \ref{prop:4.3} that $P_{DD}z_D\leq z_D$. 
Since $\rho(P_{DD})<1$, we get that $(P_{DD})^kz_D\to 0$ as $k\to \infty$, and hence 
$z_D\geq 0$. This implies that 
\begin{equation}\label{eq:4.2} 
z_C\geq P_{CC}z_C+P_{CD}z_D\geq P_{CC}z_C.
\end{equation}
Let $C_1,\ldots,C_r$ be the critical classes of $P$. 
As $P_{C_iC_i}$ is nonnegative and irreducible, it follows from the Perron-Frobenius 
theorem \cite[Theorem ??]{Gant} that there exists for each $1\leq i\leq r$ a positive 
$m^i\in\mathbb{R}^{C_i}$ such that $m^iP_{C_iC_i}=m^i$. 
Put $m=(m_1,\ldots,m_r)\in\mathbb{R}^C$  and remark that, as $P_{CC}$ is block diagonal that $mP_{CC}=m$.
Multiplication by $m$ from the left in (\ref{eq:4.2}) gives
\[
mz_C\geq mP_{CC}z_C+mP_{CD}z_D\geq mP_{CC}z_C=mz_C.
\]
As $m$ is positive, we deduce that 
\begin{equation}\label{eq:4.3}
P_{CD}z_D=0\mbox{\quad and \quad}P_{CC}z_C=z_C,
\end{equation}
which proves $(i)$. 

Recall that $z_D\geq 0$. To show $(ii)$ we assume by way of contradiction that $z_j>0$ for some $j\in D$. By definition there exists a path $(i_1,\ldots,i_q)$ in 
$\mathcal{G}(P)$ with $i=i_1\in C$, $i_2,\ldots,i_q\in D$, and $j=i_q$. 
Since $Pz\leq z$ we get that 
\[
(P_{CD}z_D)_i= P_{iD}z_D\geq P_{i_1i_2}z_{i_2}\geq P_{i_1i_2}P_{i_2i_3}z_{i_3}
\geq\ldots\geq \big{(}\prod_{k=1}^{q-1} P_{i_ki_{k-1}}\big{)}z_{i_q}>0,
\] 
which contradicts (\ref{eq:4.3}) and hence $z_D=0$.
By Proposition \ref{prop:4.3}$(i)$ we also find that  $(Pz)_{C\cup D}=z_{C\cup D}$.

To prove that $z_I\geq 0$, we remark that $z_I\geq P_{ID}z_D+P_{II}z_I=P_{II}z_I$ 
by $(ii)$. 
As $\rho(P_{II})<1$, we get that $z_I\geq P^k_{II}z_I\to 0$ as $k\to \infty$, so that 
$z_I\geq 0$. 

Finally assume that $S\subseteq\{1,\ldots,n\}$ contains at least one element in each critical class of $P$ and $z_S\gg 0$. Remark that $z_U\geq P_{UU}z_U+P_{UC}z_C+P_{UI}z_I\geq P_{UU}z_U+P_{UC}z_C$, since $z_I\geq 0$. 
As $\rho(P_{UU})<1$, we get that $(I-P_{UU})^{-1}=\sum_{k\geq 0} P_{UU}^k$ is nonnegative and 
\begin{equation}\label{eq:4.4} 
z_U\geq (I-P_{UU})^{-1}P_{UC}z_C.
\end{equation}
As each $P_{C_iC_i}$ is irreducible and $\rho(P_{C_iC_i})=1$, it follows from the 
Perron-Frobenius theorem that $z_{C_i}$ is a multiple of the unique positive eigenvector of $P_{C_iC_i}$. 
By assumption $z_{C_i}$ has at least one positive coordinate, and hence 
$z_C\geq 0$. 
It now follows from (\ref{eq:4.4}) that $z_U\geq 0$ and therefore $z\geq 0$.
\end{proof}
For $z\in\mathbb{R}^n$ and $S\subseteq\{1,\ldots,n\}$ we simply say that $z=0$ on 
$S$, or, $z$ is zero on $S$, if $z_S=0$. Similar terminology will be used for 
$z_S\geq 0$ and $z_S\gg 0$. 
Given a collection $\mathcal{P}$ of nonnegative $n\times n$ matrices  we define 
\begin{equation}\label{eq:4.5} 
\mathcal{G}(\mathcal{P})=\bigcup_{P\in\mathcal{P}} \mathcal{G}(P),
\end{equation}
and we observe that the following lemma holds. 
\begin{lemma}\label{lem:4.4bis}
If $\mathcal{P}$ is a convex set of nonnegative $n\times n$ matrices, then there 
exists $M\in\mathcal{P}$ such that $\mathcal{G}(M)=\mathcal{G}(\mathcal{P})$.  
\end{lemma}
\begin{proof} 
Since the number of edges in $\mathcal{G}(\mathcal{P})$ is finite there exist 
$\mathcal{F}\subseteq \mathcal{P}$ finite such that $\mathcal{G}(\mathcal{F})=\mathcal{G}(\mathcal{P})$. Define $M=|\mathcal{F}|^{-1}\sum_{Q\in\mathcal{F}} Q$ 
and note that $M\in\mathcal{P}$, as $\mathcal{P}$ is convex. 
Moreover, $\mathcal{G}(M)=\mathcal{G}(\mathcal{F})=\mathcal{G}(\mathcal{P})$. 
\end{proof}
For a stable nonnegative matrix $P$, we let $N^c(P)$ denote the set of critical nodes 
of $\mathcal{G}(P)$ and we let $\mathcal{G}^c(P)$ denote the restriction of $\mathcal{G}(P)$ to $N^c(P)$. For a collection of stable nonnegative $n\times n$ matrices, 
$\mathcal{P}$,  we define 
\begin{equation}\label{eq:4.6} 
N^c(\mathcal{P})=\bigcup_{P\in\mathcal{P}} N^c(P)\mbox{\quad and \quad} 
\mathcal{G}^c(\mathcal{P})=\bigcup_{P\in\mathcal{P}}\mathcal{G}^c(P).
\end{equation} 
Using these concepts we can now present the main theorem of this section. 
\begin{theorem}\label{thm:4.5} 
If $\mathcal{P}$ is a convex rectangular set of stable nonnegative $n\times n$ matrices, then there exists $M\in\mathcal{P}$ such that $\mathcal{G}^c(M)=\mathcal{G}^c(\mathcal{P})$.
\end{theorem}
\begin{proof}
The assertion is trivial if $\mathcal{G}^c(\mathcal{P})$ is empty. 
Let $\mathcal{F}$ be a finite set of matrices in $\mathcal{P}$ 
such that $\mathcal{G}^c(\mathcal{F})=\mathcal{G}^c(\mathcal{P})$. 
For each $k\in N^c(\mathcal{P})$ we let 
\[
\mathcal{Q}_k=\{P_k\colon P\in\mathcal{F}\mbox{ and } k\in N^c(P)\}.
\]
For $k\not\in N^c(\mathcal{P})$ we pick an arbitrary $P\in\mathcal{P}$ and put 
$\mathcal{Q}_k=\{P_k\}$. 
Subsequently we define an $n\times n$ nonnegative matrix $M$ by
\[
M_k=|\mathcal{Q}_k|^{-1}\sum_{q\in\mathcal{Q}_k} q\mbox{\quad for }1\leq k\leq n.
\]
As $\mathcal{P}$ is convex and rectangular, $M\in\mathcal{P}$ and hence 
$\mathcal{G}^c(M)\subseteq\mathcal{G}^c(\mathcal{P})$. 

We claim that $\mathcal{G}^c(\mathcal{P})\subseteq\mathcal{G}(M)$ by construction. 
Indeed, if $(i,j)$ is an arc in $\mathcal{G}^c(\mathcal{P})$, then there exists 
$P\in\mathcal{F}$ such that $(i,j)$ is an arc in $\mathcal{G}^c(P)$. 
This implies that $i\in N^c(P)\subseteq N^c(\mathcal{P})$ and $P_i\in\mathcal{Q}_i$. 
As $P_{ij}>0$, we get that 
\[
M_{ij} = |\mathcal{Q}_i|^{-1}\sum_{q\in\mathcal{Q}_i} q_j\geq 
\frac{p_{ij}}{|\mathcal{Q}_i|}>0.
\]
Thus, $\mathcal{G}^c(M)\subseteq\mathcal{G}^c(\mathcal{P})\subseteq\mathcal{G}(M)$, so that 
\[
\mathcal{G}^c(\mathcal{P})_{|N^c(M)}=\mathcal{G}^c(M)
\]
(Here $\mathcal{G}^c(\mathcal{P})_{|N^c(M)}$ denotes the restriction of the graph 
$\mathcal{G}^c(\mathcal{P})$ to the nodes in $N^c(M)$.) 
As $N^c(M)\subseteq N^c(\mathcal{P})$, it remains to prove that 
$N^c(\mathcal{P})\subseteq N^c(M)$ to establish the 
equality $\mathcal{G}^c(M)=\mathcal{G}^c(\mathcal{P})$. 
To show the inclusion we use the following claim.\\
\emph{Claim.} If $C$ is the set of nodes of a strongly connected component of 
$\mathcal{G}^c(\mathcal{P})$, then $\rho(M_{CC})=1$. 

If we assume the claim for the moment and take $i\in N^c(\mathcal{P})$, then there 
exists a strongly connected component $C$ in $\mathcal{G}^c(\mathcal{P})$ such that 
$i\in C$. 
Clearly there exists a class $C^*$ of $M$ such that $C\subseteq C^*$ and hence 
$1=\rho(M_{CC})\leq \rho(M_{C^*C^*})$ by Proposition \ref{prop:4.2bis} and the 
claim. 
This implies that $C^*$ is a critical class of $M$ and therefore $i\in N^c(M)$. 

To prove the claim we consider a nonlinear map 
$g\colon\mathbb{R}^C_+\to\mathbb{R}^C_+$ given by, 
\[
g_k(y)=\sup_{q\in\mathcal{Q}_k} q_Cy\mbox{\quad for $k\in C$ and }
y\in\mathbb{R}^C_+.
\] 
We begin by constructing an eigenvector $u\gg 0$ for $g$. 
As $g$ is monotone, positively homogeneous, and continuous, we can use the 
Brouwer fixed point theorem to find $u\in\mathbb{R}^C_+$, with $u\neq 0$, and $\lambda \geq 0$ such that $g(u)=\lambda u$ (see \cite[pp.152--154]{AH} or 
\cite[p.201]{KR}). 
Since $\mathcal{F}$ is finite, $\mathcal{Q}_k$ is finite, and hence the $\sup$ is attained for $u$ and $k\in C$, say by $q^k\in\mathcal{Q}_k$. 
Now let $Q$ be the $n\times n$ nonnegative matrix with $Q_k=q_k$ for all $k\in C$ and $Q_k$ is some element in $\mathcal{Q}_k$ for all $k\not\in C$. 
As $\mathcal{P}$ is rectangular, $Q\in\mathcal{P}$.    
Moreover, $Q_{CC} u =g(u)=\lambda u$ and $\rho(Q_{CC})\leq \rho(Q)\leq 1$, as 
$Q$ is stable.  
Thus, we find that $\lambda \leq 1$. 

Now note that if $(i_1,i_2)$ is an arrow in $\mathcal{G}^c(\mathcal{P})$, then there 
exists $q\in\mathcal{Q}_{i_1}$ with $q_{i_2}>0$. 
This implies that if $x\in\mathbb{R}^C_+$ and $x_{i_2}>0$, then 
\begin{equation}\label{eq:4.7}
g_{i_1}(x)\geq q_{i_2}x_{i_2}>0. 
\end{equation}
Since $C$ is a strongly connected component of $\mathcal{G}^c(\mathcal{P})$, there 
exists a path from any $i$ to any $j$ in $C$. 
Recall that $u\in\mathbb{R}^C_+$ and $u\neq 0$. 
Hence there exists $j\in C$ such that $u_j>0$. 
Now let $(i_0,i_1,\ldots,i_r)$ be a path in $\mathcal{G}^c(\mathcal{P})_{|C}$ from 
$i=i_0$ to $j=i_r$. 
By (\ref{eq:4.7}), $g_{i_{r-1}}(u)\geq q_{i_r}u_{i_r}>0$, and $g^2_{i_{r-2}}(u)\geq 
q_{i_{r-1}}g_{i_{r}}(u)\geq q_{i_{r-1}}q_{i_r}u_{i_r}>0$. 
By repeating the argument we get that 
\[
u_i= g_i^r(u)=g^r_{i_0}(u)\geq \big{(}\prod_{k=1}^r q_{i_k}\big{)}u_{i_r}>0.
\]
Thus, $u_i>0$ for all $i\in C$. 

Let $k\in C$ and $q\in\mathcal{Q}_k$. 
Then there exists $P\in\mathcal{F}$ such that $q=P_k$ and $k\in N^c(P)$. 
Moreover, there exists a critical class $C'$ of $P$ with $k\in C'\subseteq C$ and 
$\rho(P_{C'C'})=1$. 
We note that 
\begin{equation}\label{eq:4.8}
P_{C'C'}u_{C'}\leq P_{C'C}u\leq g(u)_{C'}=\lambda u_{C'}\leq u_{C'},
\end{equation} 
as $P_l\in \mathcal{Q}_l$ for all $l\in C'$. 
Since $u$ is positive on $C'$, it follows from Proposition \ref{prop:4.4}$(i)$ that 
$P_{C'C'}u_{C'}=u_{C'}$, so that $\lambda =1 $ and $P_{C'C}u=u_{C'}$ 
by (\ref{eq:4.8}). 
Therefore, if $k\in C$ and $q\in\mathcal{Q}_k$, then $q_Cu=P_{kC}u=u_k$, so that 
$M_{kC}u =u_k$ for all $k\in C$. From this we conclude that 
$M_{CC}u=u$ and hence $\rho(M_{CC})=1$, which proves the claim.
\end{proof}

\section{Tangentially stable fixed points} 
By using the results from the previous section we can now start analysing the 
t-stable fixed points of monotone convex maps. 
To begin, we have the following lemma. 
\begin{lemma}\label{lem:5.1}
If $f\colon\mathcal{D}\to\mathcal{D}$ is a convex monotone map and $v$ and $w$ are 
t-stable fixed points of $f$, then $\mathcal{G}^c(\partial f(v))=\mathcal{G}^c(\partial f(w))$.
\end{lemma}
\begin{proof}
From Propositions \ref{prop:2.1}, \ref{prop:2.2}, and \ref{prop:3.1} it follows that 
$\partial f(v)$ and $\partial f(w)$ are convex rectangular sets of stable nonnegative 
$n\times n$ matrices. 
By Theorem \ref{thm:4.5} there exists $M\in\partial f(v)$ such that 
$\mathcal{G}^c(M)=\mathcal{G}^c(\partial f(v))$. 
Moreover, 
\[
w-v= f(w)-f(v)\geq M(w-v),
\]
so that $w-v=M(w-v)$ on $C\cup D$, by Proposition \ref{prop:4.4}. (Here $C$ and $D$ 
are the critical nodes and the downstream nodes of $M$.) 
This implies that 
\[
f_i(v)-f_i(w)= (v-w)_i=M_i(v-w)\mbox{\quad for all }i\in C\cup D.
\]
From this equality we deduce that 
\begin{eqnarray*}
f_i(x)-f_i(w) & = & f_i(x)-f_i(v)+f_i(v)-f_i(w)\\
			& \geq & M_i(x-v)+M_i(v-w)\\ 
			& = & M_i(x-w)
\end{eqnarray*}
for all $x\in\mathcal{D}$ and $i\in C\cup D$. 
Thus, $M_i\in\partial f_i(w)$ for all $i\in C\cup D$.
Now let $P\in\partial f(w)$ and define $Q\in \mathcal{P}_{n,n}$ by 
\[
Q_i=\left\{\begin{array}{cl} M_i &\mbox{if }i\in C\cup D\\
                                      P_i &\mbox{otherwise.}
                                      \end{array}\right.
\] 
As $\partial f(w)$ is rectangular, $Q\in\partial f(w)$. 
Clearly $\mathcal{G}^c(f(w))\supseteq \mathcal{G}^c(Q)\supseteq
\mathcal{G}^c(M)=\mathcal{G}^c(\partial f(v))$. 
By interchanging the roles of $v$ and $w$ we obtain the desired equality.
\end{proof}
By Lemma \ref{lem:5.1} we can define for a convex monotone map $f\colon\mathcal{D}\to\mathcal{D}$ with a t-stable fixed point $v\in\mathcal{D}$, 
the set of \emph{critical nodes of} $f$ and the \emph{critical graph of} $f$ respectively by,
\[
N^c(f)=N^c(\partial f(v))\mbox{\quad and \quad} 
\mathcal{G}^c(f)=\mathcal{G}^c(\partial f(v)).
\]
\begin{lemma}\label{lem:5.2} 
Let $f\colon \mathcal{D}\to\mathcal{D}$ be a convex monotone map and let $v\in\mathcal{D}$ be a t-stable fixed point of $f$. 
Let $S\subseteq\{1,\ldots,n\}$ be a set that contains at least one node in each connected component of $\mathcal{G}^c(\partial f(v))$. 
If $w\in\mathcal{D}$ is a fixed point of $f$ and $v\leq w$ on $S$, then $v\leq w$. 
\end{lemma}
\begin{proof}
By Theorem \ref{thm:4.5} there exists $M\in \partial f(v)$ such that $\mathcal{G}^c(M)=\mathcal{G}^c(\partial f(v))$. 
Then $w-v= f(w)-f(v)\geq M(w-v)$ and $w-v\geq 0$ on $S$, so that $w-v\geq 0$ by 
Proposition \ref{prop:4.4}$(iv)$.
\end{proof}
From the previous lemma we  immediately deduce the following theorem for 
t-stable fixed points. 
\begin{theorem}\label{thm:5.3} 
Let $f\colon \mathcal{D}\to\mathcal{D}$ be a convex monotone map and let $v,w\in\mathcal{D}$ be t-stable fixed points of $f$. 
If $v=w$ on a set $S\subseteq\{1,\ldots,n\}$ that has at least one node in each strongly 
connected component of $\mathcal{G}^c(f)$, then $v=w$. In particular,  the t-stable fixed point is unique, if $N^c(f)$ is empty. 
\end{theorem} 
To analyse the geometry of the t-stable fixed point set $\mathcal{E}_t(f)$ 
and t-stable periodic points, we need some preliminary results concerning convex monotone positively homogeneous maps.  

\section{Positively homogeneous maps} 
If $h\colon\mathbb{R}^n\to\mathbb{R}^n$ is a monotone convex positively homogeneous map, then $0$ is a fixed point and we can associate to $h$ a graph $\mathcal{G}(h)$ by
\[
\mathcal{G}(h)=\mathcal{G}(\partial h(0)).
\]
If, in addition, $0$ is t-stable, then we define 
\[
A(h)=\{i\colon \mbox{there exists a path in $\mathcal{G}(h)$ from $i$ to some $j\in N^c(h)$}\}
\]
and we put $B(h)=\{1,\ldots,n\}\setminus A(h)$. 
Convex monotone positively homogeneous maps, which have $0$ as a t-stable fixed point, have the following properties. 
\begin{lemma}\label{lem:6.1} 
Let $h\colon\mathbb{R}^n\to\mathbb{R}^n$be a convex monotone positively homogeneous map, with $0$ as a t-stable fixed point. 
Write $C=N^c(h)$, $A=A(h)$, and $B=B(h)$, and identify each $x\in\mathbb{R}^n$ 
with $(x_A,x_B)\in\mathbb{R}^A\times\mathbb{R}^B$. 
Then $C\subseteq A$ and the map $h$ can be rewritten in the form 
\[
h(x_A,x_B)=(h_A(x_A,x_B),h_B(x_B)),
\]
where $h_A\colon\mathbb{R}^A\times\mathbb{R}^B\to\mathbb{R}^A$ and 
$h_B\colon \mathbb{R}^B\to\mathbb{R}^B$ are convex monotone positively homogeneous maps. Moreover, if $h^A\colon\mathbb{R}^A\to\mathbb{R}^A$ is given by,
\[
h^A(y)=h_A(y,0)\mbox{\quad for all } y\in\mathbb{R}^A,
\]
then for each $y\in\mathbb{R}^A_+$ with $y_C\gg 0$ and each $i\in A$, there exists $k\geq 1$ such that $(h^A)_i^k(y)>0$. 
\end{lemma}
\begin{proof}
Since $h$ is convex and positively hommogeneous, $h'_0=h$ and 
$h(x)=\sup_{P\in\partial h(0)} Px$ for all $x\in\mathbb{R}^n$.
Since $B=\{1,\ldots,n\}\setminus A$, we can write $h$ in the form 
\[
h(x_A,x_B)=(h_A(x_A,x_B),h_B(x_A,x_B)).
\]
We note that $P_{BA}=0$ for all $P\in\partial h(0)$. Indeed, otherwise there exists 
$j\in B$ that has access to some node $i\in A$ in $\mathcal{G}(P)$. 
But this implies that there exists a path from $j$ to a node in $C$ in $\mathcal{G}(h)$, 
as $\mathcal{G}(P)\subseteq \mathcal{G}(h)$, which contradicts $j\in B$. Since $\partial h(0)$ is rectangular, 
\[
h_B(x_A,x_B)=\sup_{P\in\partial h(0)} P_{BA}x_A+P_{BB}x_B
                       =\sup_{P\in\partial h(0)} P_{BB}x_B.
                       \] 
Thus, $h_B(x_A,x_B)$ is of the form $h_B(x_B)$, and therefore $h$ can be rewritten as 
\[
h(x_A,x_B)=(h_A(x_A,x_B),h_B(x_B)),
\]
for all $(x_A,x_B)\in\mathbb{R}^A\times\mathbb{R}^B$. 

To prove the last assertion we let $y\in\mathbb{R}^A_+$ be such that $y_C\gg 0$ and 
$i\in A$. 
By Lemma \ref{lem:4.4bis} there exists $P\in\partial h(0)$ such that $\mathcal{G}(P)=\mathcal{G}(h)$. 
We have that $h^A(z)=h_A(z,0)\geq P_{AA}z$ for all $z\in\mathbb{R}^A$. 
Hence $(h^A)^k(y)\geq (P_{AA})^k y$ for all $k\geq 1$. Since there exists a path 
from $i\in A$ to some node $j\in C$, say with length $m\geq 1$, we get that $(h^A)(y)_i\geq (P_{AA})^m_{ij} y_j>0$. 
\end{proof}
By using the previous lemma we prove the following proposition. 
\begin{proposition}\label{prop:6.2} 
If $h\colon\mathbb{R}^n\to\mathbb{R}^n$ is a convex monotone positively homogeneous map, with $0$ as a t-stable fixed point, and $h$ has a fixed point 
$v$ such that $v\gg 0$ on $A(h)$, then $v=0$ on $B(h)$. 
\end{proposition}
\begin{proof}
By Theorem \ref{thm:4.5} there exists $M\in\partial h(0)$ such that 
$\mathcal{G}^c(M)=\mathcal{G}^c(h)$. Let $C_1,\ldots, C_r$ denote the critical classes of $M$, so $C=C_1\cup\ldots\cup C_r$. By the Perron-Frobenius theorem 
\cite{Gant} there exists for each $1\leq i\leq r$ a positive eigenvector $u^i\in\mathbb{R}^{C_i}$ such that $M_{C_iC_i}u^i=u^i$. Define $u\in\mathbb{R}^n$ by 
$u_{C_i}=u^i$  for $1\leq i\leq r$ and $u_j=0$ if $j\not\in C$. Clearly $u\geq 0$, 
$u\gg 0$ on $C$, and $Mu\geq u$. This implies that 
\begin{equation}\label{eq:6.0}
h(u)\geq Mu\geq u\geq 0.
\end{equation}
As $0$ is t-stable and $h=h'_0$, we know that $\mathcal{O}(u,h)$ is bounded. 
Moreover, it follows from (\ref{eq:6.0}) that $(h^k(u))_k$ is increasing and hence $v=\lim_{k\to\infty} h^k(u)$ exists. Obviously $v$ is a fixed point of $h$ and $v\geq u$, 
so that $v\gg 0$ on $C$. Remark that $u_B=0$, because $B\cap C=\emptyset$. 
Therefore $v_B=0$ by Lemma \ref{lem:6.1} and hence $v_A$ is a fixed point of $h^A$. 
By the second part of Lemma \ref{lem:6.1} we obtain that $v_A\gg 0$, as $v_C\gg 0$.
\end{proof}
For a positively homogeneous map $h\colon\mathbb{R}^n\to\mathbb{R}^n$ the 
\emph{spectral radius} is defined by 
\begin{equation}\label{eq:6.1}
\tau(h)=\sup\{\lambda \geq 0\colon h(x)=\lambda x\mbox{ for some }
x\in\mathbb{R}^n_+\setminus\{0\}\}. 
\end{equation}
\begin{proposition}\label{prop:6.3}
Let $h\colon\mathbb{R}^n\to\mathbb{R}^n$ be a convex monotone positively homogeneous map,with $0$ as a t-stable fixed point. 
If $h_B\colon\mathbb{R}^B\to\mathbb{R}^B$ is as in Lemma \ref{lem:6.1}, then 
$\tau(h_B)<1$. 
\end{proposition}
\begin{proof}
Assume by way of contradiction that $\tau(h_B)=r\geq 1$. Then there exists $v\geq 0$ 
with $v\neq 0$ such that $h_B(v)=rv\geq v$. 
But $h_B(v)=\sup_{P\in\partial h(0)} P_{BB}v$ and $\partial h(0)$ is a rectangular compact set of stable nonnegative matrices. Hence $h_B(v)=Q_{BB}v=rv$ for some $Q\in\partial h(0)$. This implies that a class $K$ of $Q$ such that $K\subseteq B$ and 
$\rho(Q_{KK})=r\geq 1$. As $Q$ is stable, $r=1$, and hence $K\subseteq N^c(Q)\subseteq N^c(h)\subseteq A$, which is a contradiction.  
\end{proof}
It is shown by Nussbaum \cite[Theorem 3.1]{LAA} that $\tau(h)=\tau'(h)$, where 
$\tau'(h)$ is the so called \emph{Collatz-Wielandt spectral radius} of a monotone positively homogeneous map $h\colon\mathbb{R}^n\to\mathbb{R}^n$, which is given by 
\begin{equation}\label{eq:6.2}
\tau'(h)=\inf\{\mu>0\colon h(x)\leq \mu x\mbox{ for some }x\gg 0\}.
\end{equation}
Thus, Proposition \ref{prop:6.3} has the following consequence. 
\begin{corollary}\label{cor:6.4}
If $h\colon\mathbb{R}^n\to\mathbb{R}^n$ is a convex monotone  positively homogeneous map, with $0$ as a t-stable fixed point, and $h_B$ is as in Lemma \ref{lem:6.1}, then there exist $0<\lambda <1$ and $w\gg 0$ such that 
$h_B(w)\leq \lambda w$.
\end{corollary}

We conclude this section by showing that a monotone convex positively homogeneous map with $0$ as a t-stable fixed point is non-expansive with respect to a polyhedral norm. Recall that  a norm on $\mathbb{R}^n$ is called \emph{polyhedral} 
if its unit ball is a polyhedron. 
\begin{theorem}\label{thm:6.5}
Let $h\colon\mathbb{R}^n\to\mathbb{R}^n$ be a monotone convex positively homogeneous map with $0$ as a t-stable fixed point, then there exist $v\gg 0$ and  $\alpha>0$ such that $h$ is non-expansive with respect to the polyhedral norm, 
\begin{equation}\label{eq:6.3}
\|x\|_v=\max_{i\in A(h)} |x_i/v_i|+\alpha\max_{i\in B(h)} |x_i/v_i|\mbox{\quad for }x\in\mathbb{R}^n,
\end{equation}
where $A(h)$ and $B(h)$ are as in Lemma \ref{lem:6.1}.
\end{theorem}
\begin{proof}
We use the same notation as in Lemma \ref{lem:6.1}. 
By Proposition  \ref{prop:6.2} $h$ has an eigenvector $u\in\mathbb{R}^n$ such that 
$u\gg 0$ on 
$A=A(h)$ and $u=0$ on $B=B(h)$. Moreover, by Corollary \ref{cor:6.4} there also exists $0<\lambda<1$ and $w\in\mathbb{R}^B$ such that $w\gg 0$ and 
$h_B(w)\leq \lambda w$. Let $v\in\mathbb{R}^A\times\mathbb{R}^B$ be defined by $v=u$ on $A$ and $v=w$ on $B$. 
Further let $W$ be the diagonal matrix with $v$ as its diagonal and define 
$g\colon\mathbb{R}^n\to\mathbb{R}^n$ by $g(x)=(W^{-1}\circ h\circ W)(x)$ for all 
$x\in\mathbb{R}^n$. 
It follows from Lemma \ref{lem:6.1} that we can write $g$ in the form
\[
g(x)=(g_A(x_A,x_B),g_B(x_B)),
\]
where $g_A\colon\mathbb{R}^A\times\mathbb{R}^B\to\mathbb{R}^B$ and 
$g_B\colon\mathbb{R}^B\to\mathbb{R}^B$ are convex monotone positive homogeneous maps. Moreover, $g_A(\mu \mathds{1},0)=\mu\mathds{1}$ and 
$g_B(\mu\mathds{1})\leq \lambda\mu\mathds{1}$ for all $\mu\geq 0$, where $\mathds{1}$ denotes the vector with all coordinates unity. 

We write $\|\cdot\|_\infty$ to denote the sup-norm, so $\|x\|_\infty=\max_i|x_i|$. 
Remark that 
\[
g(x)_B-g(y)_B\leq g(x-y)_B = g_B((x-y)_B)\leq \lambda \|(x-y)_B\|_\infty\mathds{1},
\]
as $g$ is convex, monotone and positively homogeneous. 
By interchanging the roles of $x$ and $y$ we deduce that 
\begin{equation}\label{eq:6.4}
\|g(x)_B-g(y)_B\|_\infty\leq \lambda\|x_B-y_B\|_\infty\mbox{\quad for all }x,y\in\mathbb{R}^n.
\end{equation}
Subsequently we remark that there exists $C>0$ finite such that 
$g_A(0,x_B)\leq C\|x_B\|_\infty\mathds{1}$, as $g_A$ is continuously and positively homogeneous. 
(Recall that $g(x)=\sup_{P\in\partial g(0)} Px$ for each $x\in\mathbb{R}^n$.) 
Thus, for each $x\in\mathbb{R}^n$,  
\begin{equation}\label{eq:6.5}
g(x)_A= g_A(x_A,x_B)\leq g_A(x_A,0)+g_B(0,x_B)\leq 
\|x_A\|_\infty\mathds{1} + C\|x_B\|_\infty\mathds{1}.
\end{equation}

As $g$ is convex, 
\begin{equation}\label{eq:6.6} 
g(x)\leq g(x-y) +g(y),
\end{equation}
so that $g(x)-g(y)\leq g(x-y)$ for all $x,y\in\mathbb{R}^n$. 
As $g(0)=0$, we have that $-g(y)\leq g(-y)$, and hence it follows from (\ref{eq:6.5}) that 
\begin{equation}\label{eq:6.7} 
\|g(x)_A\|_\infty \leq \|x_A\|_\infty+C\|x_B\|_\infty 
\end{equation}
for all $x\in\mathbb{R}^n$. It also follows from (\ref{eq:6.6}) that 
\begin{equation}\label{eq:6.8}
\|g(x)_A-g(y)_A\|_\infty\leq \max\{ \|g(x-y)_A\|_\infty,\|g(y-x)_A\|_\infty\}.
\end{equation}
Now let $\alpha>C/(1-\lambda)$ and define $\|\cdot\|'$ on $\mathbb{R}^n$ by 
\[
\|x\|'=\|x_A\|_\infty +\alpha\|x_B\|_\infty\mbox{\quad for all }x\in\mathbb{R}^A\times\mathbb{R}^B.
\]
It now follows from (\ref{eq:6.4}), (\ref{eq:6.7}) and (\ref{eq:6.8}) that 
\begin{eqnarray*}
\|g(x)-g(y)\|' & = &  \|g(x)_A-g(y)_A\|_\infty +\alpha \|g(x)_B-g(y)_B\|_\infty\\
                    & \leq & \|(x-y)_A\|_\infty +C\|(x-y)_B\|_\infty +
                    \alpha\lambda\|(x-y)_B\|_\infty\\
                    & \leq & \|(x-y)_A\|_\infty + \alpha\|(x-y)_B\|_\infty\\
                    & = & \|x-y\|'.
\end{eqnarray*}
Finally, we recall that $g\circ W^{-1}= W^{-1}\circ g$, so that $h$ is non-expansive 
with respect to $\| W^{-1}(\cdot)\|'=\|\cdot\|_v$ and we are done.
\end{proof}

\section{The tangentially stable fixed point set} 
Throughout this section we assume, in addition to Hypothesis \ref{hyp:2.3}, that the domain $\mathcal{D}\subseteq\mathbb{R}^n$ satisfies the following property. 
\begin{hypothesis}\label{hyp:7.0} 
The domain $\mathcal{D}\subseteq\mathbb{R}^n$ is a \emph{downward} set, i.e., 
if $x\in\mathcal{D}$ and $y\leq x$, then $y\in\mathcal{D}$. 
\end{hypothesis}
Given a convex monotone map $f\colon\mathcal{D}\to\mathcal{D}$ we define $\mathcal{E}^+(f)=\{z\in\mathcal{D}\colon f(z)\leq z\}$. 
As $f$ is convex, $\mathcal{E}^+(f)$ is convex. There exists a natural projection 
from $\mathcal{E}^+(f)$ onto $\mathcal{E}(f)$ if $f$ has a t-stable fixed point 
(cf.\ \cite[Lemma 3.3]{AG}). 
\begin{lemma}\label{lem:7.1} 
If $f\colon \mathcal{D}\to\mathcal{D}$ is a convex monotone map with a t-stable 
fixed point, then 
\begin{equation}\label{eq:7.0} 
f^\omega(z)=\lim_{k\to\infty} f^k(z)
\end{equation}
exists and $f^\omega(z)=z$ on $N^c(f)$ for each $z\in\mathcal{E}^+(f)$. In addition, the  map 
$f^\omega\colon\mathcal{E}^+(f)\to\mathcal{E}(f)$ is a surjective convex monotone projection, i.e., $(f^\omega)^2=f^\omega$.
\end{lemma} 
\begin{proof}
Since $f\colon\mathcal{D}\to\mathcal{D}$ has a t-stable fixed point, all orbits of $f$ are bounded from below by Proposition \ref{prop:3.1}. 
Therefore, $f^\omega(z)=\lim_{k\to\infty} f^k(z)$ exists for all $z\in\mathcal{E}^+(f)$, 
as $(f^k(z))_k$ is a decreasing sequence and $\mathcal{D}$ is downward. 
By continuity of $f$, $f^\omega(z)$ is a fixed point of $f$. 

Let $v$ be a t-stable fixed point of $f$ and $z\in\mathcal{E}^+(f)$. 
By Theorem \ref{thm:4.5} there exists $Q\in\partial f(v)$ such that $\mathcal{G}^c(Q)=\mathcal{G}^c(\partial f(v))=\mathcal{G}^c(f)$. 
We also have that 
\begin{equation}\label{eq:7.1}
z-v\geq f(z)-v=f(z)-f(v)\geq Q(z-v).
\end{equation}
From Proposition \ref{prop:4.4} it follows that $z-v=Q(z-v)$ on $N^c(f)=N^c(Q)$ and hence $z=f(z)$ on $N^c(f)$. Replacing $z$ by $f^k(z)$ in the previous argument gives 
$f^{k+1}(z)=f^k(z)=\ldots = z$ on $N^c(f)$ for all $k\geq 1$. 
Thus, $f^\omega(z)=\lim_{k\to\infty} f^k(z)=z$ on $N^c(f)$. 
Clearly, $f^\omega(x)=x$ if $x\in\mathcal{E}(f)$, so that 
$f^\omega\colon\mathcal{E}^+(f)\to\mathcal{E}(f)$ is onto and 
$(f^\omega)^2=f^\omega$. Moreover, as $f^\omega$ is the pointwise limit of $(f^k)_k$, 
$f^\omega$ is a convex monotone map. 
\end{proof}
The fixed point set $\mathcal{E}(f)$ can be naturally equipped with a binary operation 
$\wedge_f$ that turns $(\mathcal{E}(f),\wedge_f)$ into an inf-semilattice, if 
$f\colon\mathcal{D}\to\mathcal{D}$ has a t-stable fixed point. 
The relation $\wedge_f$ is $\mathcal{E}(f)$ is defined by 
\begin{equation}\label{eq:7.2} 
x\wedge_f y=\lim_{k\to\infty} f^k(x\wedge y).
\end{equation}
We note  that if $x,y\in\mathcal{E}(f)$, then $f(x\wedge y)\leq f(x)=x$ and 
$f(x\wedge y)\leq f(y)=y$, so that $f(x\wedge y)\leq x\wedge y$. 
As $f\colon\mathcal{D}\to\mathcal{D}$ has all its orbits bounded from below and 
$\mathcal{D}$ is downward, the limit (\ref{eq:7.2}) exists. 
To prove that $(\mathcal{E}(f),\wedge_f)$, is an inf-semilattice one has to show that 
$\wedge_f$ is associative, symmetric, and reflexive, which is a simple exercise.  
It also follows from Lemma \ref{lem:7.1} that if we put $C=N^c(f)$ and define 
$r_C\colon\mathcal{E}(f)\to \mathbb{R}^C$ by $r_C(x)=x_C$ for all 
$x\in\mathcal{E}(f)$, then $r_C(\mathcal{E}(f))$ is an inf-semilattice in 
$\mathbb{R}^C$, where $\wedge$ is the infimum operation induced by the partial 
ordering $\leq$ on $\mathbb{R}^C$. 
Indeed, if $x,y\in\mathcal{E}(f)$ and $v\in\mathcal{D}$ is a t-stable fixed point of $f\colon\mathcal{D}\to\mathcal{D}$, then there exists $M\in\partial f(v)$ such that 
$\mathcal{G}^c(M)=\mathcal{G}^c(f)$. As $f(x\wedge y)\leq f(x)=x$ and $f(x\wedge y)\leq f(y)=y$, $f(x\wedge y)\leq x\wedge y$, so that $f^k(x\wedge y)\leq f^{k-1}(x\wedge y)$ for all $k\geq 1$. This implies that 
\[
f^{k-1}(x\wedge y)-v\geq f^k(x\wedge y)-f(v)\geq M(f^{k-1}(x\wedge y)-v).
\] 
By Proposition \ref{prop:4.4} we get that 
\[
f^{k-1}(x\wedge y)-v=f^k(x\wedge y) -f(v)=f^k(x\wedge y)-v
\]
on $C$ for all $k\geq 1$. 
Hence 
\[
r_C(x\wedge_f y)= (x\wedge_f y)_C=\lim_{k\to\infty} f^k(x\wedge y)_C=
(x\wedge y)_C = r_C(x)\wedge r_C(y), 
\]
so $(r_C(\mathcal{E}(f)),\wedge)$ is an inf-semilattice in $\mathbb{R}^C$. 
The difference between $(\mathcal{E}(f),\wedge_f)$ and 
$(r_C(\mathcal{E}(f)),\wedge)$ is illustrated by the following simple example. 
Consider 
\[
P=\left (\begin{matrix} 0 & 1/2 & 1/2\\ 0 & 1 & 0\\ 0 & 0 & 1\end{matrix}\right ). 
\]
So, $P$ is a projection.  
Clearly, $\mathcal{E}(P)=\mathrm{span}\{(1/2,1,0),(1/2,0,1)\}$, 
but $\mathcal{E}(P)$ is not 
an inf-semilattice with respect to $\wedge$, as $(1/2,0,0)=(1/2,1,0)\wedge (1/2,0,1)$ 
is for instance not in $\mathcal{E}(P)$. In this case $N^c(P)=\{2,3\}$ and 
$\mathrm{span}\{(1,0),(0,1)\}$ is an inf-semilattice with respect to $\wedge$. 

Let us now analyse the t-stable fixed point set in more detail. 
We shall prove the following theorem. 
\begin{theorem}\label{thm:7.2} 
If $f\colon \mathcal{D}\to\mathcal{D}$ is a convex monotone map with a t-stable 
fixed point, then $(\mathcal{E}_t(f),\wedge_f)$ is an  inf-semilattice and 
$(r_C(\mathcal{E}_t(f)),\wedge)$ is a convex inf-semilattice in $\mathbb{R}^C$, 
where $C=N^c(f)$.
\end{theorem}
But first we give two preliminary lemmas. 
\begin{lemma}\label{lem:7.3} 
Let $f\colon \mathcal{D}\to\mathcal{D}$ be a convex monotone map with a 
t-stable fixed point. Let $z\in\mathcal{E}^+(f)$ and $w=f^\omega(z)$ be given by (\ref{eq:7.0}). Write $S=\{i\colon w_i<z_i\}$ and $E=\{i\colon w_i=z_i\}$. 
Then the map $h=f'_w$ can be written in the form 
\begin{equation}\label{eq:7.3} 
h(x_S,x_E)=(h_S(x_S,x_E),h_E(x_E)),
\end{equation}
where $h_S\colon\mathbb{R}^S\times\mathbb{R}^E\to\mathbb{R}^S$ and $h_E\colon\mathbb{R}^E\to\mathbb{R}^E$ are convex positively homogeneous maps. Moreover, the map $h^S\colon\mathbb{R}^S\to\mathbb{R}^S$ given by, $h^S(y)=h_S(y,0)$ for 
$y\in\mathbb{R}^S$, satisfies $\tau(h^S)<1$.
\end{lemma}
\begin{proof}
Since $h(x)=\sup_{P\in\partial f(w)} Px$ for all $x\in\mathbb{R}^n$ and $f(z)\leq z$, we get that 
\begin{equation}\label{eq:7.4} 
z-w\geq f(z)-w\geq h(z-w).
\end{equation}
As $z-w\geq 0$ and $z=w$ on $E$, we find that $h(z-w)=0$ on $E$. 
Now suppose that  that $y\in\mathbb{R}^n$ is such that $y_E=0$. 
Then there exists $\lambda >0$ such that $y\leq \lambda (z-w)$, as $z-w\gg 0$ 
on $S$. 
This implies that $h(y)\leq \lambda h(z-w)\leq \lambda (z-w)$, and hence $h(y)\leq 0$ 
on $E$ if $y=0$ on $E$. 
Now let $y_S,y'_S\in\mathbb{R}^S$ and $y_E\in\mathbb{R}^E$, and remark that 
\[
h(y_S,y_E)_E\leq h(y_S-y'_S,0)_E+h(y'_S,y_E)_E\leq h(y'_S,y_E)_E,
\]
as $h$ is positively homogeneous and convex. 
Thus, $h(y_S,y_E)_E$ is independent of $y_S$ and hence $h$ can written in the form 
(\ref{eq:7.3}).

Let $h^S\colon\mathbb{R}^S\to\mathbb{R}^S$ be given by $h^S(y)=h_S(y,0)$ 
for all $y\in \mathbb{R}^S$ and denote $v=(z-w)_S$. 
Then $v\gg 0$ and by (\ref{eq:7.4}), 
\[
v\geq h(z-w)_S\geq h^S(v).
\] 
This implies that $\tau(h^S)\leq 1$, as $\tau(h^S)=\tau'(h^S)$ by 
\cite[Theorem 3.1]{LAA}. 
(Here $\tau(\cdot)$ and $\tau'(\cdot)$ are as in (\ref{eq:6.1}) and (\ref{eq:6.2}), respectively. 

Assume by way of contradiction that $\tau(h^S)=1$. Then there exists $u\in\mathbb{R}^S_+$ such that $u\neq 0$ and $h^S(u)=u$. 
Then $\eta=(u,0)\in\mathbb{R}^S\times \mathbb{R}^E$ satisfies $h(\eta)=\eta$ by (\ref{eq:7.3}). 
Since $\partial f(w)$ is a compact rectangular set of nonnegative matrices and 
$h( \eta)=\sup_{P\in\partial f(w)} P\eta$, there exists $P\in\partial f(w)$ such that 
$\eta = h(\eta)=P\eta$. 
Let $S'=\{i\colon \eta_i\neq 0\}$ and remark that $S'\subseteq S$ and $S'$ is a union 
of classes of $P$. To proceed we need the notion of a final class. 
A class of a nonnegative matrix is called final if it has no access to any other class. 
It is known (see \cite[Theorem 3.10]{BP}) that a nonnegative matrix $M$ has a positive eigenvector if, and only if, each final class of $M$ is basic. 
Clearly $P_{S'S'}$ has a final class, say $F$. 
As $P_{S'S'}\eta_{S'}=\eta_{S'}$ and $\eta_{S'}\gg 0$, we find that $F$ is a basic class 
of $P_{S'S'}$, and hence $\rho(P_{FF})=\rho(P_{S'S'})=1$. 
By (\ref{eq:7.4}) we have that 
\begin{equation}\label{eq:7.5}
(z-w)_F\geq (f(z)-w)_F\geq h(z-w)_F\geq (P(z-w))_F\geq P_{FF}(z-w)_F.
\end{equation}
By the Perron-Frobenius theorem there exists $m\gg 0$ in $\mathbb{R}^F$ such that 
$mP_{FF} = m $. This implies that $m(z-w)_F\geq mP_{FF}(z-w)_F = m(z-w)_F$, and 
hence $(z-w)_F = P_{FF}(z-w)_F$.  Thus, $f(z)_F=z_F$ by (\ref{eq:7.5}). 
Similarly we deduce that $f^k(z)_F=z_F$ for all $k\geq 1$. 
Indeed, 
\[
z-w\geq f^k(z)-w\geq (f^k)'_w(z-w)=(f'_w)^k(z-w)=h^k(z-w)
\]
and 
\[
h^k(z-w)_F\geq (P^k)_{FF}(z-w)_F\geq (P_{FF})^k(z-w)_F.
\]

Recall that $w_F=\lim_{k\to\infty} f^k(z)_F$ and therefore $w_F=z_F$. 
But this implies that $F\subseteq E$, which contradicts the fact that $F\subseteq S'\subseteq S$. Thus, we conclude that $\tau(h^S)<1$.
\end{proof}
\begin{lemma}\label{lem:7.4} 
Let $h$, $h_E$, $h^S$, $S$ and $E$ be as in Lemma \ref{lem:7.3}. 
If $h_E\colon\mathbb{R}^E\to\mathbb{R}^E$ has all its orbits bounded from above, then $h$ has all its orbits bounded from above.
\end{lemma}
\begin{proof}
Since $h\colon\mathbb{R}^n\to\mathbb{R}^n$ is monotone, it suffices to prove that 
$\mathcal{O}(x;h)$ is bounded from above for all $x\in\mathbb{R}^n_+$. 
As $h$ can be written in the form (\ref{eq:7.3}), we know that 
$\{h^k(x)_E\colon k\geq 0\}=\{h_E^k(x_E)\colon k\geq 0\}$ is bounded from above. 
It therefore remains to be shown that $\{h^k(x)_S\colon k\geq 0\}$ is bounded 
from above. 
Since $\tau'(h^S)=\tau(h^S)<1$, there exist $u\in\mathbb{R}^S_+$ and $0<\alpha <1$ such that $u\gg 0$ and $h^S(u)\leq \alpha u$. 
For $y\in\mathbb{R}^S$ we define a norm by 
\[
\|y\|_u=\max_{i\in S} |y_i/u_i|.
\]
For each $y\in\mathbb{R}^S_+$ we have that $y\leq \|y\|_u u$, so that  
\[
\|h^S(y)\|_u\leq \|h^S(\|y\|_u u)\|_u = \|h^S(u)\|_u\|y\|_u\leq \alpha\|y\|_u,
\]
as $h^S$ is positively homogeneous and monotone. 
Since $\{h^k(x)_E\colon k\geq 0\}$ is bounded from above, there exists $v\gg 0$ 
in $\mathbb{R}^E$ such that $h^k(x)_E\leq v$ for all $k\geq 0$. 
This implies that 
\[
h(0,h^k(x)_E)_S\leq h(0,v)_S\leq \gamma u
\]
for some $\gamma >0$. 
Now using the fact that $h$ is positively homogeneous and convex, we get that 
\[
0\leq h^{k+1}(x)_S\leq h(h^k(x)_S,0)_S+h(0,h^k(x)_E)_S\leq h^S(h^k(x)_S) 
+\gamma u, 
\] 
so that 
\[
\|h^{k+1}(x)_S\|_u\leq \|h^S(h^k(x)_S)\|_u +\gamma\leq 
\alpha \|h^k(x)_S\|_u +\gamma.
\]
By induction we obtain 
\[
\|h^{k+1}(x)_S\|_u\leq \alpha^k\|x_S\|_u +\frac{\gamma}{1-\alpha},
\]
which shows that $\{h^k(x)_S\colon k\geq 0\}$ is bounded from above. 
\end{proof}
Let us now prove Theorem \ref{thm:7.2} 
\begin{proof}[Proof of Theorem \ref{thm:7.2}]
To prove that $(\mathcal{E}_t(f),\wedge_f)$ is an inf-semilattice, it suffices to show that if $x,y\in\mathcal{E}_t(f)$, then $x\wedge_f y\in\mathcal{E}_t(f)$, as $(\mathcal{E}(f),\wedge_f)$ is an inf-semilattice. 
So, suppose that $x,y\in\mathcal{E}_t(f)$. 
Put $z=x\wedge y$ and let $w= x\wedge_f y$. 
We need to show that $h=f'_w\colon\mathbb{R}^n\to\mathbb{R}^n$ has all its orbits bounded from above by Proposition \ref{prop:3.1}. 
By Lemma \ref{lem:7.3} we can write $h$ in the form of (\ref{eq:7.3}), since $f(z)\leq z$. 
We also get that $\tau(h^S)<1$.   
By Lemma \ref{lem:7.4} it is sufficient to prove that $h_E\colon\mathbb{R}^E\to
\mathbb{R}^E$ has all its orbit bounded from above. 
For each $i\in E$ we have that 
\[
f_i(w)=w_i=x_i\wedge y_i=f_i(x)\wedge f_i(y),
\]
as $x,y,w\in\mathcal{E}(f)$. Note that $w\leq x\wedge y$ implies that 
\[
w+\epsilon u\leq (x\wedge y)+\epsilon u= (x+\epsilon u)\wedge (y+\epsilon u)
\]
for all $u\in\mathbb{R}^n$ and $\epsilon >0$. Now let $u\in\mathbb{R}^n$ and $\epsilon >0$ sufficiently small so that $x+\epsilon u, y+\epsilon u\in\mathcal{D}$. 
Then 
\[
f_i(w+\epsilon u)\leq f_i(x+\epsilon u)\wedge f_i(y+\epsilon u), 
\] 
since $f_i$ is monotone, and hence 
\begin{eqnarray*}
f_i(w+\epsilon u)-f_i(w)& \leq & 
f_i(x+\epsilon u)\wedge f_i(y+\epsilon u)- f_i(x)\wedge f_i(y)\\
 & \leq & \max\{ f_i(x+\epsilon u)-f_i(x),  f_i(y+\epsilon u)- f_i(y)\} 
\end{eqnarray*}
for all $i\in E$. This implies that 
\begin{equation}\label{eq:7.6} 
f'_w(u)_i=\lim_{\epsilon \downarrow 0} \frac{f_i(w+\epsilon u)-f_i(w)}{\epsilon}
\leq \max\{ f'_x(u)_i,f'_y(u)_i\}
\end{equation}
for all $i\in E$. 

Applying the same argument for the map $f^k$ and using the fact that 
$(f^k)'_v=(f'_v)^k$ for all $v\in\mathcal{E}(f)$, we obtain that 
\[
(f'_w)^k(u)_i\leq \max \{(f'_x)^k(u)_i, (f'_y)^k(u)_i\}
\] 
for all $u\in\mathbb{R}^n$, $i\in E$ and $k\geq 1$. 
But $x$ and $y$ are t-stable and therefore 
\[
(h_E)^k(s)=(h^k(0,s))_E = (f'_w)^k(0,s)_E
\]
is bounded from above as $k\to\infty$ for all $s \in \mathbb{R}^E$. 
Thus, we conclude that $w$ is a t-stable fixed point of $f$. 

Recall that $(r_C(\mathcal{E}(f)),\wedge)$ is an inf-semilattice in $\mathbb{R}^C$, where $C=N^c(f)$, and $r_C(x\wedge_f y)=r_C(x)\wedge r_C(y)$ for all $x,y\in\mathcal{E}(f)$. 
This implies that $(r_C(\mathcal{E}_t(f)),\wedge)$ is a  inf-semilattice in 
$\mathbb{R}^C$. 
To show that $r_C(\mathcal{E}_t(f))$ is convex we apply the same technique as before. 

Let $x,y\in\mathcal{E}_t(f)$ and $0<\lambda < 1$. 
Put $z=\lambda x+(1-\lambda)y$. Since $f$ is convex, 
\[
f(z)\leq \lambda f(x)+(1-\lambda)f(y)=\lambda x+(1-\lambda)y =z.
\]
Hence $w=f^\omega(z)=\lim_{k\to\infty} f^k(z)$ exists and is a fixed point of $f$ 
with $w_C=z_C=\lambda x_C+(1-\lambda)y_C$. 
We need to show that $w$ is t-stable. 
Let $h=f'_w$ and recall that is suffices to 
show that $h_E$ all its orbit bounded from above by Lemma \ref{lem:7.4}. 
Let $i\in E$  and note that as $x,y,w\in\mathcal{E}(f)$, 
\[
f^k_i(w) = w_i =\lambda x_i+(1-\lambda)y_i=\lambda f_i^k(x)+(1-\lambda)f_i^k(y).
\]
Clearly $w\leq \lambda x+(1-\lambda)y$, so that 
\[
w+\epsilon u\leq \lambda(x+\epsilon u)+(1-\lambda)(y+\epsilon u)
\]
for all $u\in\mathbb{R}^n$ and $\epsilon >0$. Now fix $u\in\mathbb{R}^n$. Then 
for all $\epsilon >0$ sufficiently small, we have that 
\[
f^k_i(w+\epsilon u)\leq f^k_i(\lambda(x+\epsilon u)+(1-\lambda)(y+\epsilon u))
\leq \lambda f^k_i(x+\epsilon u) +(1-\lambda)f_i^k(y+\epsilon u)
\]
for all $k\geq 1$. Thus, 
\[
f_i^k(w+\epsilon u)-f^k_i(w)\leq \lambda(f^k_i(x+\epsilon u)-f_i^k(x))
+(1-\lambda)(f^k_i(y+\epsilon u)-f^k_i(y)) 
\]
and hence 
\begin{eqnarray*}
(f^k)'_w(u)_i &\leq & \lambda (f^k)'_x(u)_i+(1-\lambda)(f^k)'_y(u)_i\\
                     & \leq &  \lambda (f'_x)^k(u)_i+(1-\lambda)(f'_y)^k(u)_i
\end{eqnarray*}
for all $k\geq 1$ and $i\in E$. As $x,y\in\mathcal{E}_t(f)$, the right hand side 
is bounded from above as $k\to\infty$. 
Therefore $((f'_w)^k(0,s)_E)_k=((h_E)^k(s))_k$ 
is bounded from above for all $s\in\mathbb{R}^E$, which shows that  $w$ is 
t-stable. 
\end {proof}
We note that if  $\mathcal{E}_t(f)$ is compact, then it is a connected set. To show this it suffices to prove that $r_C^{-1}$ is continuous, as $r_C(\mathcal{E}_t(f))$ is convex. 
Note that $r_C$ is one-to-one on $\mathcal{E}_t(f)$ by Theorem  \ref{thm:5.3}. So, let $y_k\to y$ in $r_C(\mathcal{E}_t(f))$, $r_C(x_k)=y_k$ for all $k$, and $r_C(x)=y$. 
If $(x_{k_i})_i$ is a subsequence of $(x_k)_k$ and  $x_{k_i}\to z$, then $z\in\mathcal{E}_t(f)$ by compactness. Moreover, $r_C(z)=y$, which implies that $z_C=x_C$. Thus, by Theorem \ref{thm:5.3}, $z=x$, and hence $r_C^{-1}$ is continuous. 
 
Another consequence of Theorem \ref{thm:7.2} is the following. Recall that $\mathcal{O}(\xi)=\{\xi, f(\xi),\ldots,f^{p-1}(\xi)\}\subseteq \mathcal{D}$ is a \emph{Lyapunov stable periodic orbit} of $f$ if for all  neighbourhoods $U_i$ of $f^i(\xi)$, $i=0,\ldots,p-1$, there exist  neighbourhoods $V_i$ of $f^i(\xi)$, $i=0,\ldots,p-1$, such that $f^{kp+i}(y)\in U_i$ for all $y\in V_i$ for all $k\geq 0$.    
\begin{corollary}\label{cor:7.4.1}
If $f\colon\mathcal{D}\to\mathcal{D}$ is a convex monotone map with a t-stable periodic point $\xi\in \mathcal{D}$, then $f$ has a t-stable fixed point. Moreover, 
if $\xi$ has a Lyapunov stable orbit, then $f$ has a Lyapunov stable fixed point.   
\end{corollary}
\begin{proof}
Let $\xi\in\mathcal{D}$ be a t-stable periodic point of $f$ with period $p$. Put $g= f^p$ and note that $f^k(\xi)$ is a t-stable fixed point of $g$ for all $0\leq k<p$. 
Let $z=\xi\wedge f(\xi)\wedge \cdots \wedge f^{p-1}(\xi)$. Clearly $g(z)\leq z$ and hence $u:=g^\omega(z)$ exists and is a t-stable fixed point of $g$ by Theorem \ref{thm:7.2}. As $f^k(z)\leq z$ for all $0\leq k<p$, we have that 
\[
f^k(u) =f^k(g^\omega(z))=g^\omega(f^k(z))\leq u
\]
for all $0\leq k<p$. In particular, $f(u)\leq u$ and 
\[
u=g(u)= f^p(u)\leq f(u) \leq u,
\]
so that $f(u)=u$. Moreover, as $(f'_u)^p = g'_u$ and $f'_u$ is continuous, we conclude that $u$ is a t-stable fixed point of  $f$.  

Now assume that $\xi$ has a Lyapunov stable orbit. For each $i=0,\ldots,p-1$ there exists a neighbourhood $V_i$ of $f^i(\xi)$ such that the orbit of each $y\in V_i$ is bounded 
from above.  Let $W_i=\{x\in\mathcal{D}\colon x\leq y\mbox{ for some }y\in V_i\}$ and put $W =\cup_{i=0}^{p-1} W_i$. Note that, as $f^i(\xi)\in V_i$, $u\in W$ and $W$ is a neighbourhood of $u$. As $f$ is monotone and the orbit of each $y\in V_i$, $i=0,\ldots,p-1$, is bounded from above under $f$, the orbit of each $w\in W$ is bounded from above, and hence $u$ is a Lyapunov stable fixed point.  
\end{proof}
We also have the following result. 
\begin{lemma}\label{lem:7.4.2}
Suppose $f\colon\mathcal{D}\to\mathcal{D}$ is a convex monotone map and $v$ and $w$ are  fixed points with $w\leq v$. If $w\ll v$ or $v$ Lyapunov stable,  then $w$ is Lyapunov stable.  
\end{lemma}
\begin{proof}
Suppose that $w\ll v$ and let $V =\{x\in \mathcal{D}\colon x\ll z$ be an open neighbourhood of $w$, where $w\ll z\ll v$. For each $x\in V$ we have that $x\ll v$, so that $f^k(x)\leq f^k(v)=v$ for all $k\geq 0$. Thus, the orbit of $x$ is bounded from above and hence $w$ is Lyapunov stable by Proposition \ref{prop:3.1}.

Now assume that $w\leq v$ and $v$ is Lyapunov stable. Then there exists a neighbourhood $U$ of $v$ such that the orbit of each $y\in U$ is bounded from above.  
Let $W=\{x\in\mathcal{D}\colon x\leq y\mbox{ for some }y\in U\}$. Note that $W$ is a neighbourhood of $w$, since $v\in U$ and $w\leq v$. As $f$ is monotone, the orbit of each $x\in W$ is bounded from above under $f$, and hence $w$ is Lyapunov stable. 
\end{proof}

We conclude this section by showing that every t-stable fixed point of a convex monotone piece-wise affine map is Lyapunov stable. Recall that a  map $f\colon\mathbb{R}^n\to\mathbb{R}^n$ is \emph{piece-wise affine} if $\mathbb{R}^n$ can be partitioned into polyhedra such that the restriction of $f$ to each polyhedron is an affine map. 
\begin{corollary}\label{cor:7.5}
If $f\colon\mathbb{R}^n\to\mathbb{R}^n$ is a convex monotone piece-wise affine map, then every t-stable fixed point of $f$ is  Lyapunov stable.
\end{corollary}
\begin{proof}
Since $f$ is piece-wise affine, we can find by \cite[Lemma 6.4]{AG} 
a neighborhood $\mathcal{W}$ of $0$ such that 
$f(v+x)=f(v)+f'_v(x)$ for all $x\in \mathcal{W}$. 
By Theorem \ref{thm:6.5} there exists a norm under which $f'_v$ is non-expansive. Since $f'_v(0)=0$, we can take any open ball, $\mathcal{B}$, around $0$ for this 
norm, and get $f'_v(\mathcal{B})\subseteq \mathcal{B}$.
By taking $\mathcal{B}$ of sufficiently small radius, we can guarantee
that $\mathcal{B}\subseteq \mathcal{W}$.
Since $f(v)=v$, we find for all $x\in \mathcal{B}$ that 
$f(v+x)=v+f'_v(x)\in v+\mathcal{B}$, which shows that
$f(v+\mathcal{B})\subseteq v+\mathcal{B}$. Since this inclusion
holds for all balls $\mathcal{B}$ of sufficiently small
radius, $v$ is a Lyapunov stable fixed point of $f$.
\end{proof}

\section{Tangentially stable periodic points}
For a directed graph $\mathcal{G}$ and integer $k\geq 1$ we let $\mathcal{G}^k$ 
be the directed graph that has the same nodes as $\mathcal{G}$ and it has an arrow 
from node $i$ to  node $j$ if, and only if, there exists a directed path of length $k$ in 
$\mathcal{G}$ from $i$ to $j$. There exists the following relation between 
$\mathcal{G}^c(f^k)$ and $(\mathcal{G}^c(f))^k$. 
\begin{theorem}\label{thm:8.1}
If $f\colon\mathcal{D}\to\mathcal{D}$ is a convex monotone map with a t-stable 
fixed point, then $\mathcal{G}^c(f^k)=(\mathcal{G}^c(f))^k$ for all $k\geq 1$.  
\end{theorem} 
To prove this theorem we reduce it to a special case, which was analysed in \cite{AG}. 
Recall that $g\colon\mathbb{R}^n\to\mathbb{R}^n$ is called \emph{additively homogeneous} if $g(x+\lambda \mathds{1})=g(x)+\lambda\mathds{1}$ for all 
$x\in\mathbb{R}^n$ and $\lambda\in\mathbb{R}$. (Here $\mathds{1}$ is the 
vector in $\mathbb{R}^n$ with all coordinates unity.)
The map $g$ is said to be \emph{additively subhomogeneous} if 
$g(x+\lambda \mathds{1})\leq g(x)+\lambda\mathds{1}$ for all $x\in\mathbb{R}^n$ 
and $\lambda \geq 0$. 
The following theorem  for convex monotone additively homogeneous maps is proved in \cite[Theorem 4.1]{AG}.
\begin{theorem} [\cite{AG}]\label{thm:8.2} 
If $g\colon\mathbb{R}^n\to\mathbb{R}^n$ is a convex monotone additively homogeneous map with a fixed point, then $\mathcal{G}^c(g^k)=(\mathcal{G}^c(g))^k$ for all $k\geq 1$.  
\end{theorem}
We note that every fixed point of a monotone subhomogeneous map 
$g\colon\mathbb{R}^n\to\mathbb{R}^n$ is stable, because $g$ is non-expansive with respect to the sup-norm \cite{CT}. 
Using a standard "cemetery state" argument, see \cite[Section 1.4]{AG}, we derive the following consequence of Theorem \ref{thm:8.2}. 
\begin{corollary}\label{cor:8.3} 
If $g\colon\mathbb{R}^n\to\mathbb{R}^n$ is a convex monotone additively subhomogeneous map with a fixed point, then 
$\mathcal{G}^c(g^k)=(\mathcal{G}^c(g))^k$ for all $k\geq 1$.  
\end{corollary}
\begin{proof}
Define $h\colon\mathbb{R}^{n+1}\to\mathbb{R}^{n+1}$ by 
\[
h(x,x_{n+1})=\left(\begin{array}{c} x_{n+1}\mathds{1}+g(x-x_{n+1}\mathds{1})\\x_{n+1}
\end{array}\right)
\]
for all $(x,x_{n+1})\in\mathbb{R}^{n+1}$. It is easy to verify that $h$ is a convex monotone additively homogeneous map. 
Let $v\in\mathbb{R}^n$ be  a fixed point of $g$ and remark that $w=(v,0)$ is a fixed point of $h$. Due to the triangular structure of $h$, we see that 
$\partial g(v)=\partial h(w)_{JJ}$, where $J=\{1,\ldots,n\}$ and $\mathcal{G}^c(h)$ is the union of $\mathcal{G}^c(g)$ with the loop $\{(n+1),(n+1)\}$. 

The same is true for the critical graph of $g^k$ and $h^k$, as 
\[
h^k(x,x_{n+1})=\left (\begin{array}{c} x_{n+1}\mathds{1}+g^k(x-x_{n+1}\mathds{1})
\\x_{n+1}
\end{array}\right )
\]
for all $(x,x_{n+1})\in\mathbb{R}^{n+1}$. 
It follows from Theorem \ref{thm:8.2} that $\mathcal{G}^c(h^k)=(\mathcal{G}^c(h))^k$ 
for all $k\geq 1$. By considering the subgraph on the nodes $\{1,\ldots,n\}$ we find that 
 $\mathcal{G}^c(g^k)=(\mathcal{G}^c(g))^k$  for all $k\geq 1$.
\end{proof}
We shall use the previous corollary to prove the following proposition, which is the key 
ingredient in the proof of Theorem \ref{thm:8.1}. 
\begin{proposition}\label{prop:8.4} 
If $f\colon\mathcal{D}\to\mathcal{D}$ is a convex monotone map with a 
t-stable fixed point $v\in\mathcal{D}$, then 
$\mathcal{G}^c((f'_v)^k)=(\mathcal{G}^c(f'_v))^k$ for all $k\geq 1$.  
\end{proposition}
\begin{proof}
Recall that $f'_v\colon\mathbb{R}^n\to\mathbb{R}^n$ is a convex monotone 
positively homogeneous map. Write $h=f'_v$ and let $A$, $B$, $C$, $h_A$, $h_B$ and $h^A$ be as in Lemma \ref{lem:6.1}. 
Similar notation will be used for $h^k$, so $(h^k)_A$, $(h^k)_B$ and $(h^A)^k$. 
By definition of $A$ we know that $\mathcal{G}^c(h^A)=\mathcal{G}^c(h)$. 
As $v$ is a t-stable fixed point of $f$, $0$ is t-stable for $h$. 
It follows from Lemma \ref{lem:6.1} that we can write $h$ in the form 
\[
h(x_A,x_B)=(h_A(x_A,x_B),h_B(x_B)),
\]
where $h_A\colon\mathbb{R}^A\times\mathbb{R}^B\to\mathbb{R}^A$ and 
$h_B\colon\mathbb{R}^B\to\mathbb{R}^B$ are convex monotone positively homogeneous maps. 
Due to the this form we have that $(h^k)_B=(h_B)^k$ and $(h^k)^A=(h^A)^k$, where  
$h^A\colon\mathbb{R}^A\to\mathbb{R}^A$ is given by $h^A(x_A)=h_A(x_A,0)$. 
This implies that $\mathcal{G}^c((h^A)^k)=\mathcal{G}^c(h^k)$. 

By Proposition \ref{prop:6.2} there exists $v_A\gg 0$ in $\mathbb{R}^A$ such that $h^A(v_A)=v_A$. 
Let $W$ be the $|A|\times|A|$ diagonal matrix with $W_{ii}=(v_A)_i$ for all $i\in A$. 
Now define $g\colon\mathbb{R}^A\to\mathbb{R}^A$ by 
\[
g(x)=(W^{-1}\circ h^A\circ W)(x)\mbox{\quad for all }x\in\mathbb{R}^A.
\]
Clearly $\mathcal{G}^c((h^A)^k)=\mathcal{G}^c(g^k)$ for all $k\geq 1$. 
Moreover, $g(0)=0$ and $g(\mathds{1})=\mathds{1}$. so that 
\[
g(x+\lambda \mathds{1})\leq g(x)+\lambda \mathds{1}
\]  
for all $x\in\mathbb{R}^A$ and $\lambda \geq 0$, as $g$ is convex and positively homogeneous. 
By Corollary \ref{cor:8.3} we get that 
$\mathcal{G}^c(g^k)=(\mathcal{G}^c(g))^k$ for all $k\geq 1$. 
From this it follows that 
\[
\mathcal{G}^c(h^k)=\mathcal{G}^c((h^A)^k)=\mathcal{G}^c(g^k)=(\mathcal{G}^c(g))^k 
= (\mathcal{G}^c(h^A))^k=(\mathcal{G}^c(h))^k,
\]
which completes the proof.
\end{proof}
By using Proposition \ref{prop:8.4} it is now easy to prove Theorem \ref{thm:8.1} 
\begin{proof}[Proof of Theorem \ref{thm:8.1}]
Let $f\colon\mathcal{D}\to\mathcal{D}$ be a convex monotone map with a t-stable fixed point $v\in\mathcal{D}$. 
Then $0$ is a t-stable fixed point of $f'_v$. 
For each $k\geq 1$ we have that $\partial f^k(v)=\partial (f^k)'_v(0)$, so that 
\[
\mathcal{G}^c(f^k)=\mathcal{G}^c(\partial f^k(v))=\mathcal{G}^c(\partial (f^k)'_v(0)) =
\mathcal{G}^c((f^k)'_v).
\]
Thus, it follows from Proposition \ref{prop:8.4} that 
\[
\mathcal{G}^c(f^k)=\mathcal{G}^c((f^k)'_v)=\mathcal{G}^c((f'_v)^k)=
(\mathcal{G}^c(f'_v))^k=(\mathcal{G}^c(f))^k
\]
for each $k\geq 1$, and we are done.
\end{proof}

To analyse the periods of t-stable periodic points we need to recall the notion 
of cyclicity of a graph. The \emph{cyclicity} of a strongly connected directed graph 
$\mathcal{G}$, denoted $c(\mathcal{G})$, is the greatest common divisor of the
lengths of its circuits. The cyclicity of a general directed graph $\mathcal{G}$ is given by, 
\[
c(\mathcal{G})=\mathrm{lcm}\,\{c(\mathcal{G}_i)\colon \mbox{$\mathcal{G}_i$ is a strongly connected component of $\mathcal{G}$}\}.
\] 
The \emph{cyclicity} of a nonnegative stable matrix $P$ is defined by 
$c(P)=c(\mathcal{G}^c(P))$. Note that $c(P)$ is the order of a permutation on $n$ letters, where $n$ is the size of the matrix. 
The following consequence of the Perron-Frobenius theorem concerning the cyclicity  of stable nonnegative matrices can be found in \cite[Theorem 9.1]{NVL}.
\begin{theorem}[\cite{NVL}]\label{thm:8.5}
If $P$ is stable nonnegative matrix, then the period of each periodic point of $P$ divides $c(P)$.  
\end{theorem}
For a convex monotone map $f\colon\mathcal{D}\to\mathcal{D}$ with a t-stable 
fixed point, we define the \emph{cyclicity of} $f$ by $c(f)=c(\mathcal{G}^c(f))$. 
\begin{theorem}\label{thm:8.6}
If $\mathcal{D}\subseteq\mathbb{R}^n$ is downward and $f\colon\mathcal{D}\to\mathcal{D}$ is a convex monotone map with a t-stable fixed point, then the 
period of  each t-stable periodic point of $f$ divides $c(f)$. In particular, 
the period of each t-stable periodic point is the order of a permutation on $n$ letters.  
\end{theorem}
\begin{proof}
Let $v\in\mathcal{D}$ be a t-stable fixed point of $f$ and let $\xi\in \mathcal{D}$  
be a t-stable periodic point of $f$ with period $p$. 
Put $g=f^{c(f)}$ and note that $c(g)=1$. 
Indeed, by Theorem \ref{thm:8.1} we get that $\mathcal{G}^c(g)=\mathcal{G}^c(f^{c(f)})=(\mathcal{G}^c(f))^{c(f)}=\cup_{i=1}^s \mathcal{G}_i^{c(f)}$, where $\mathcal{G}_1,\ldots,\mathcal{G}_s$ are the disjoint strongly connected components of 
$\mathcal{G}^c(f)$. 
Let $c_i=c(\mathcal{G}_i)$ for $1\leq i\leq s$ and note that $c_i$ divides $c(f)$. 
It well-known that if $\mathcal{G}_i$ is a strongly connected graph with cyclicity $c_i$, then $c(\mathcal{G}_i^{kc_i})=1$ for all $k\geq 1$ (see \cite[Section 2]{BP}). 
In particular, we get that $c(\mathcal{G}_i^{c(f)})=1$ for all $1\leq i\leq s$. 
Thus, 
\[
c(g)=c(\cup_{i=1}^s \mathcal{G}_i^{c(f)})=\mathrm{lcm}\,\{c(\mathcal{G}_i^{c(f)})\colon 1\leq i\leq s\}=1.
\]

By Theorem \ref{thm:4.5} there exists  $M\in\partial g(v)$  such that 
$\mathcal{G}^c(M)=\mathcal{G}^c(g)$, so $c(M)=c(g)=1$. 
Now let $C$, $D$, $U$, and $I$ as in Proposition \ref{prop:4.3}. 
As $g(v)=v$ and $g(x)-g(v)\geq M(x-v)$ 
for all $x\in\mathcal{D}$, we get that 
\begin{equation}\label{eq:8.1}
g^k(x)\geq M^k(x-v) +v\mbox{\quad for all }x\in\mathcal{D}\mbox{ and }k\geq 1.
\end{equation} 
In particular, 
\begin{equation}\label{eq:8.1.1}
\xi =g^p(\xi)\geq M^p(\xi-v)+v,
\end{equation}
so that $\xi-v\geq M^p(\xi-v)$. 
It follows that $\xi-v=M^p(\xi-v)$ on $C\cup D$ and $\xi-v=0$ on $D$ from 
Proposition \ref{prop:4.4}. 
This implies that $g^p(\xi)\geq \xi$ on $C\cup D$. 

Put $F=C\cup D$ and $G=U\cup I$. By definition we have that $M_{FG}=0$, and so 
$M^k_{FG}=0$ for all $k\geq 1$. From this we deduce that 
\[
(\xi-v)_F=(M^p(\xi-v))_F=(M_{FF})^p(\xi-v)_F=M^p_{FF}(\xi-v)_F.
\] 
The  matrix $M$ is stable and has cyclicity one. Therefore the matrix $M_{FF}$ is 
also stable and has cyclicity one. This implies that any periodic point of $M_{FF}$ must have period one by Theorem \ref{thm:8.5}. Thus, we find that 
$(\xi-v)_F=M_{FF}(\xi-v)_F$, so that $(\xi-v)_F=M^k_{FF}(\xi-v)_F$ for all $k\geq 1$. 
Since $M_{FG}=0$ we deduce that 
\[
(\xi-v)=M^k(\xi-v)\mbox{\quad on }F.
\]

From (\ref{eq:8.1}) it follows that $g^k(\xi)\geq \xi$ on $F$. 
Let $z=\xi\wedge g(\xi)\wedge\ldots\wedge g^{p-1}(\xi)$. 
Clearly $z\leq \xi$ and $z=\xi$ on $F$. As $g(z)\leq z$, it follows from 
Lemma \ref{lem:7.1} that $g^\omega(z)=\lim_{k\to\infty} g^k(z)$ exists 
and $g^\omega(z)=z$ on $C$. Thus, $g^\omega(z)=\xi$ on $C$. 
Note that $g^\omega(z)$ and $\xi$ are fixed points of $g^p$, and $\xi$ is a t-stable fixed point of $g^p=f^{p c(f)}$. Therefore it follows from Lemma \ref{lem:5.2} that $g^\omega(z)\geq \xi$. As $g^\omega(z)\leq z\leq\xi$, we conclude that $g^\omega(z)=z=\xi$. Hence $f^{c(f)}(\xi)=g(\xi)=\xi$ from which we conclude that the period of 
$\xi$ divides $c(f)$.  
\end{proof}
Remark that if $f$ is strongly monotone convex map with a t-stable fixed point $v\in\mathcal{D}$, then every $P\in\partial f(v)$ is positive by Proposition \ref{prop:2.2}. 
Hence $c(f)=1$ in that case, and  therefore $f$ has no t-stable periodic points except its t-stable fixed points. 
We also like to point out that in Theorem \ref{thm:8.6} the t-stability 
assumptions are essential. Indeed, consider
\[
A=\left(\begin{array}{ccc}
\cos \alpha&\sin \alpha &0\\
-\sin \alpha &\cos \alpha &0\\
0&0&b\end{array}\right)
\mbox{\quad and \quad}
P=
\left(\begin{array}{ccc}
1&0&1\\ 0&1&1\\ -\alpha&-\alpha&1
\end{array}\right),
\]
then
\[
B=PAP^{-1} \sim
\left(\begin{array}{ccc}
1&\alpha b &b-1\\ \alpha(b-2)&1&b-1\\ \alpha(b-1) &\alpha(b-1)&b
\end{array}\right)
\]
when $\alpha>0$ is sufficiently small. 
Thus, for $b>2$, the matrix $B$ is nonnegative for all sufficiently small $\alpha > 0$.
Now if $\alpha=2\pi/p$, then $A$ has a periodic point of period $p$ and hence $B$ 
has one too.
This shows that a monotone convex map may have (unstable) periodic orbits
with arbitrary large periods. We also remark that if $\alpha>0$ is an irrational 
multiple of $\pi$, then $B$ has an unstable bounded orbit that does not converge to a 
periodic orbit.   

\section{Global convergence and non-expansiveness}
In this final section we give a condition under which every orbit of a convex  monotone 
map $f\colon\mathcal{D}\to\mathcal{D}$, where $\mathcal{D}=\mathbb{R}^n$,  converges to a Lyapunov stable periodic orbit. 
To present it we need the notion of the \emph{recession map} $\hat{f}$ of $f$, which can be defined by 
\[
\hat{f}(x)=\lim_{\lambda\to\infty} \frac{1}{\lambda} f(\lambda x)\mbox{\quad for all }
x\in\mathbb{R}^n.
\]
Since $f\colon\mathbb{R}^n\to\mathbb{R}^n$ is convex, 
$\hat{f}(x)\in (\mathbb{R}\cup \{\infty\})^n$ exists for each $x\in\mathbb{R}^n$ and is equal to 
\begin{equation}\label{eq:9.1} 
\hat{f}(x)=\sup_{y\in\mathbb{R}^n} f(y+x)-f(y)
\end{equation}
(see \cite[Theorem 8.2]{Rock}). 
As $\hat{f}$ is the point-wise limit of a convex monotone map, $\hat{f}$ is also convex and monotone. 
The next theorem shows that if $\hat{f}$ has all its orbits bounded from above, then $f$ is non-expansive with respect to a polyhedral norm. 
\begin{theorem}\label{thm:9.1} 
If $f\colon\mathbb{R}^n\to\mathbb{R}^n$ is a convex monotone map and the recession map $\hat{f}$ has all its orbits bounded from above, then $f$ is non-expansive with respect to the norm $\|\cdot\|_v$ given in Theorem \ref{thm:6.5}. 
\end{theorem} 
\begin{proof} We remark that $\hat{f}$ is a convex, monotone, and positively homogeneous map, which has $0$ as a t-stable fixed point, because $(\hat{f})'_0=\hat{f}$ and $\hat{f}$ has all its orbits bounded from above. 
Moreover, 
\begin{equation}\label{eq:9.2} 
-\hat{f}(-x)\leq f(y+x)-f(y)\leq \hat{f}(x)
\end{equation}
for all $x\in\mathbb{R}^n$ by (\ref{eq:9.1}). 

Let $\|\cdot\|_v$ be the polyhedral norm from Theorem \ref{thm:6.5} and remark that 
$\hat{f}$ is non-expansive with respect to $\|\cdot\|_v$. 
Clearly $\|u\|_v\leq \|w\|_v$ if $u,w\in\mathbb{R}^n$ are such that $0\leq |u|\leq |w|$, where $|z|=(|z_1|.\,\ldots,|z_n|)$.  Therefore it follows from (\ref{eq:9.2}) that 
\[
\|f(y+x)-f(y)\|_v\leq \max\{\|\hat{f}(-x)\|_v,\|\hat{f}(x)\|_v\}\leq \|x\|_v
\] 
for all $x,y\in\mathbb{R}^n$. Thus, $f$ is also non-expansive with respect to 
$\|\cdot\|_v$.
\end{proof}
This theorem has the following consequence.
\begin{corollary}\label{cor:9.2}
If $f\colon\mathbb{R}^n\to\mathbb{R}^n$ is a convex monotone  map
and $f$ is non-expansive with respect to some norm on
$\mathbb{R}^n$, then $f$ is non-expansive with respect to a
polyhedral norm.
\end{corollary}
\begin{proof}
We note that $\hat{f}(x)=\lim_{\lambda\to\infty} f(\lambda x)/\lambda$ for all 
$x\in\mathbb{R}^n$.
As $f$ is non-expansive with respect to some norm, $\hat{f}$ will be non-expansive 
with respect to the same norm.
This implies that $\hat{f}$ has all its orbits bounded from above, since $\hat{f}(0)=0$.
From Theorem \ref{thm:9.1} we conclude that $f$ is non-expansive with respect to a 
polyhedral norm.
\end{proof}
It is proved in \cite{N1}  that if a map $f\colon\mathbb{R}^n\to\mathbb{R}^n$ is non-expansive with respect to a polyhedral norm,
that every bounded orbit of $f$ converges to a periodic orbit.
Moreover, if the unit ball of the polyhedral norm has $N$ facets, then  the period of each periodic point of a non-expansive map does not exceed 
$\max_k 2^k{\lfloor N/2\rfloor \choose k}$, see \cite{LS2}. 
By using these results, the following global convergence theorem can be proved.
\begin{theorem}\label{thm:9.3}
If $f\colon\mathbb{R}^n\to\mathbb{R}^n$ is a convex monotone map, with a
fixed point, and the recession map $\hat{f}$ has all its orbit
bounded from above, then every orbit of $f$ converges to a
Lyapunov stable periodic orbit of $f$ whose period divides $c(f)$. 
\end{theorem}
\begin{proof}
As $\hat{f}$ has all its orbits bounded from above, we know by
Theorem \ref{thm:9.1} that $f$ is non-expansive with respect to a
polyhedral norm, so that all periodic orbits of $f$ are Lyapunov stable. 
Since $f$ has a fixed point, it follows from \cite{N1} that
every orbit of $f$ converges to a periodic orbit.  Proposition \ref{prop:3.1} 
implies that  every periodic point of $f$ is t-stable, and hence 
the period of each periodic point divides $c(f)$ by Theorem \ref{thm:8.6}, which completes the proof. 
\end{proof}

 \footnotesize

\bibliography{convex}
\bibliographystyle{plain}

\end{document}

%% file: graph.tex
\begin{picture}(0,0)%
\includegraphics{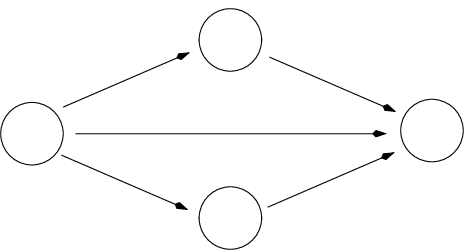}%
\end{picture}%
\setlength{\unitlength}{1973sp}%
\begingroup\makeatletter\ifx\SetFigFontNFSS\undefined%
\gdef\SetFigFontNFSS#1#2#3#4#5{%
  \reset@font\fontsize{#1}{#2pt}%
  \fontfamily{#3}\fontseries{#4}\fontshape{#5}%
  \selectfont}%
\fi\endgroup%
\begin{picture}(4456,2324)(1643,-4178)
\put(3721,-3964){\makebox(0,0)[lb]{\smash{{\SetFigFontNFSS{10}{12.0}{\rmdefault}{\mddefault}{\updefault}{\color[rgb]{0,0,0}$I$}%
}}}}
\put(3721,-2236){\makebox(0,0)[lb]{\smash{{\SetFigFontNFSS{10}{12.0}{\rmdefault}{\mddefault}{\updefault}{\color[rgb]{0,0,0}$C$}%
}}}}
\put(1801,-3136){\makebox(0,0)[lb]{\smash{{\SetFigFontNFSS{10}{12.0}{\rmdefault}{\mddefault}{\updefault}{\color[rgb]{0,0,0}$U$}%
}}}}
\put(5701,-3136){\makebox(0,0)[lb]{\smash{{\SetFigFontNFSS{10}{12.0}{\rmdefault}{\mddefault}{\updefault}{\color[rgb]{0,0,0}$D$}%
}}}}
\end{picture}%